\documentclass[11pt]{article}
\usepackage{latexsym,amsfonts,amssymb,amsmath,amsthm}
\usepackage{graphicx}

\usepackage[usenames,dvipsnames]{color}
\usepackage{ulem}

\begin{document}
\setlength{\baselineskip}{16pt}

\parindent 0.5cm
\evensidemargin 0cm \oddsidemargin 0cm \topmargin 0cm \textheight 22cm \textwidth 16cm \footskip 2cm \headsep
0cm

\newtheorem{theorem}{Theorem}[section]
\newtheorem{lemma}{Lemma}[section]
\newtheorem{proposition}{Proposition}[section]
\newtheorem{definition}{Definition}[section]
\newtheorem{example}{Example}[section]
\newtheorem{corollary}{Corollary}[section]

\newtheorem{remark}{Remark}[section]

\numberwithin{equation}{section}

\def\p{\partial}
\def\I{\textit}
\def\R{\mathbb R}
\def\C{\mathbb C}
\def\u{\underline}
\def\l{\lambda}
\def\a{\alpha}
\def\O{\Omega}
\def\e{\epsilon}
\def\ls{\lambda^*}
\def\D{\displaystyle}
\def\wyx{ \frac{w(y,t)}{w(x,t)}}
\def\imp{\Rightarrow}
\def\tE{\tilde E}
\def\tX{\tilde X}
\def\tH{\tilde H}
\def\tu{\tilde u}
\def\d{\mathcal D}
\def\aa{\mathcal A}
\def\DH{\mathcal D(\tH)}
\def\bE{\bar E}
\def\bH{\bar H}
\def\M{\mathcal M}
\renewcommand{\labelenumi}{(\arabic{enumi})}

\def\disp{\displaystyle}
\def\undertex#1{$\underline{\hbox{#1}}$}
\def\card{\mathop{\hbox{card}}}
\def\sgn{\mathop{\hbox{sgn}}}
\def\exp{\mathop{\hbox{exp}}}
\def\OFP{(\Omega,{\cal F},\PP)}
\newcommand\JM{Mierczy\'nski}
\newcommand\RR{\ensuremath{\mathbb{R}}}
\newcommand\CC{\ensuremath{\mathbb{C}}}
\newcommand\QQ{\ensuremath{\mathbb{Q}}}
\newcommand\ZZ{\ensuremath{\mathbb{Z}}}
\newcommand\NN{\ensuremath{\mathbb{N}}}
\newcommand\PP{\ensuremath{\mathbb{P}}}
\newcommand\abs[1]{\ensuremath{\lvert#1\rvert}}

\newcommand\normf[1]{\ensuremath{\lVert#1\rVert_{f}}}
\newcommand\normfRb[1]{\ensuremath{\lVert#1\rVert_{f,R_b}}}
\newcommand\normfRbone[1]{\ensuremath{\lVert#1\rVert_{f, R_{b_1}}}}
\newcommand\normfRbtwo[1]{\ensuremath{\lVert#1\rVert_{f,R_{b_2}}}}
\newcommand\normtwo[1]{\ensuremath{\lVert#1\rVert_{2}}}
\newcommand\norminfty[1]{\ensuremath{\lVert#1\rVert_{\infty}}}
\newcommand{\ds}{\displaystyle}

\title{Spreading Speeds and Traveling Waves of Nonlocal Monostable Equations
 in Time and Space Periodic Habitats\thanks{Partially supported by NSF grant DMS--0907752}}

\author{Nar Rawal and Wenxian Shen\\
Department of Mathematics and Statistics\\
Auburn University\\
Auburn University, AL 36849\\
U.S.A. \\
\\
and\\
\\
Aijun Zhang\\
Department of Mathematics\\
University of Kansas\\
Lawrence, KS 66045\\
U.S.A.}

\date{}
\maketitle

\noindent {\bf Abstract.}
This paper is devoted to the investigation of spatial spreading speeds and traveling wave solutions of monostable
evolution equations with nonlocal dispersal in time and space periodic habitats. It has been shown in an earlier work by the first two authors of the current paper
that  such an equation has a unique time and space periodic positive stable solution $u^*(t,x)$. In this paper, we show that such an equation has a spatial spreading speed
$c^*(\xi)$ in the direction of any given unit vector $\xi$. A variational characterization of $c^*(\xi)$ is  given.
Under the assumption that the nonlocal dispersal operator associated to the linearization of the monostable equation at the trivial solution
$0$ has a principal eigenvalue, we also show that the monostable equation has a continuous periodic traveling wave solution connecting $u^*(\cdot,\cdot)$ and $0$ propagating in
any given direction of $\xi$ with speed $c>c^*(\xi)$.

\bigskip

\bigskip


\newpage

%\newpage

\section{Introduction}
\setcounter{equation}{0}

In 1937,  Fisher \cite{Fisher} and Kolmogorov, Petrowsky,
Piscunov \cite{KPP} independently studied the following reaction diffusion equation modeling the evolutionary take-over of a habitat by a fitter genotype,
\begin{equation}
\label{classical-fisher-eq}
 \frac{\p u}{\p t}=\frac{\p ^ 2u}{\p x^2}+u(1-u),\quad\quad x\in \RR.
\end{equation}
 Here $u$ is the frequency
of one of two forms of a gene. Fisher in \cite{Fisher} found traveling wave solutions $u(t,x)=\phi(x-ct)$,
$(\phi(-\infty)=1,\phi(\infty)=0)$ of all speeds $c\geq 2$ and showed that there are no such traveling wave
solutions of slower speed. He conjectured that the take-over occurs at the asymptotic speed $2$. This conjecture
was proved in \cite{KPP} by Kolmogorov, Petrowsky, and Piscunov, that is, they proved that for any nonnegative
solution $u(t,x)$ of \eqref{classical-fisher-eq}, if at time $t=0$, $u$ is $1$ near $-\infty$ and $0$ near
$\infty$, then $\lim_{t\to \infty}u(t,ct)$ is $0$ if $c>2$ and $1$ if $c<2$ (i.e. the population invades into
the region  with no initial population  with speed
$2$). The number $2$ is called the {\it spatial spreading speed} of \eqref{classical-fisher-eq} in  literature.

The results of Fisher \cite{Fisher} and Kolmogorov, Petrowsky,
Piscunov \cite{KPP}  for \eqref{classical-fisher-eq} have been extended by many people to quite general reaction diffusion equations of the form,
\begin{equation}
\label{general-fisher-eq}
u_t=\Delta u+u f(t,x,u),\quad x\in\RR^N,
\end{equation}
where $f(t,x,u)<0$ for $u\gg 1$,  $\p_u f(t,x,u)<0$ for $u\ge 0$, and $f(t,x,u)$ is of certain recurrent property in $t$ and $x$.
 For example, assume  that $f(t,x,u)$ is periodic in
$t$ with period $T$ and periodic in $x_i$  with period $p_i$ ($p_i>0$, $i=1,2,\cdots,N$) (i.e. $f(\cdot+T,\cdot,\cdot)=f(\cdot,\cdot+p_i{\bf e_i},\cdot)=f(\cdot,\cdot,\cdot)$,
${\bf e_i}=(\delta_{i1},\delta_{i2}, \cdots,\delta_{iN})$, $\delta_{ij}=1$ if $i=j$ and $0$ if $i\not =j$,
$i,j=1,2,\cdots,N$), and that
$u\equiv 0$ is a linearly unstable solution of \eqref{general-fisher-eq} with respect to periodic perturbations. Then it is known that
\eqref{general-fisher-eq} has a unique positive periodic solution $u^*(t,x)$ ($u^*(t+T,x)=u^*(t,x+p_i{\bf e_i})=u^*(t,x)$)  which is asymptotically stable with
respect to periodic perturbations   and
 it
has been proved
 that for every $\xi\in S^{N-1}:=\{x\in \RR^N\,|\, \|x\|=1\}$,
  there is a $c^*(\xi)\in\RR$ such that for every $c\geq c^*(\xi)$, there is a
 traveling wave solution connecting $u^*$ and $u\equiv 0$ and propagating in the direction of $\xi$
 with speed $c$, and there is no such traveling
 wave solution of slower speed in the direction of $\xi$.  Moreover, the minimal wave speed $c^*(\xi)$ is
 of some  important spreading  properties.  For spreading properties,
the reader is referred to  \cite{Wei1} for  homogeneous equations, to \cite{BeHaNadin}, \cite{Fre}, \cite{FrGa}, and \cite{Wei2}
for  periodic equations, and to \cite{BeHaNadin}, \cite{BeHaNa2} for general equations. About existence of traveling waves and characterization of their speeds,
the reader is referred to
\cite{BeHaNa1},  \cite{BeHaRo} for space periodic equations, and
to \cite{Nad}, \cite{NoRuXi}, \cite{NoXi1}, \cite{Wei2} for space-time periodic equations. The reader is referred to
 \cite{LiZh} and \cite{LiZh1} for spreading properties and existence of traveling waves in homogeneous and periodic systems, respectively, and
 to
\cite{HuSh},  \cite{NaRo}, \cite{She1}, \cite{She2} for the extensions of the above results to the cases that $f(t,x,u)$ is almost periodic in $t$ and periodic in $x$
and that $f(t,x,u)\equiv f(t,u)$ is recurrent in $t$.

Among others, equation \eqref{general-fisher-eq} is used to model the evolution of population density of a species with random
internal interaction or movement
among the organisms
(roughly,  the organisms move randomly between the adjacent spatial locations). The term $\Delta u$  in \eqref{general-fisher-eq}
characterizes the internal interaction or movement of the organisms  and is sometime referred to as {\it random dispersal}.
 In practice, the internal interaction or movement among the organisms in many biological systems
is not local. Evolution equations of the following form are widely used to model such systems,
\begin{equation}
\label{main-eq}
 \frac{\p u}{\p
 t}=\int_{\RR^N}k(y-x)u(t,y)dy-u(t,x)+u(t,x)f(t,x,u(t,x)),\quad
 x\in\RR^N,
\end{equation}
 where
 $k(\cdot)$ is a $C^1$
 %symmetric
 convolution kernel supported on a ball centered at $0$ (i.e.
  $k(z)>0$ if $\|z\|<r_0$ and $k(z)=0$ if $\|z\|\geq r_0$ for some $r_0>0$, where
 $\|\cdot\|$ denotes the norm
 in $\RR^N$), and  $\int_{\RR^N}k(z)dz=1$ (see \cite{BaZh},  \cite{ChChRo}, \cite{Fif2}, \cite{GrHiHuMiVi}, \cite{HuMaMiVi},
 etc.). In \eqref{main-eq}, the term
 $\int_{\RR^N}k(y-x)u(t,y)dy-u(t,x)$ characterizes the internal interaction or movement of the organisms
 and is sometime referred to  as {\it nonlocal dispersal}.

Recently,  nonlocal dispersal equations of form \eqref{main-eq} have  been studied by many authors. See,
for example,  \cite{CoDaMa1},   \cite{KaLoSh},
 \cite{ShZh0}, \cite{ShZh2} for the study of   the existence, uniqueness,
and stability of  positive  stationary solutions of \eqref{main-eq} in the case that $f(t,x,u)\equiv f(x,u)$ is spatially periodic.
 See, for example, \cite{Cov1} and \cite{CoDu}
for the study of traveling waves  of \eqref{main-eq} in the case that $f(t,x,u)\equiv f(u)$
  is homogeneous,
  and  \cite{CoDaMa2} and
\cite{ShZh1}  for the study of traveling waves
 of \eqref{main-eq}  in the case that $f(t,x,u)\equiv f(x,u)$ is spatially periodic. See also \cite{HeShZh}, \cite{ShZh0}, and \cite{ShZh2}
 for the study of spreading properties of \eqref{main-eq} with $f(t,x,u)=f(x,u)$ being spatially periodic.
However, in contrast to  \eqref{general-fisher-eq}, the dynamics of \eqref{main-eq} with both time and space periodic dependence
or with  general time and/or space dependence
is much less understood.
The results on spatial spreading speeds and traveling wave solutions  established in \cite{LiZh1} and \cite{Wei2} for quite general
periodic monostable evolution equations cannot be applied to time and space periodic nonlocal monostable equations because of the lack of certain
compactness of the solution operators for such equations.

The objective of the current paper is to  explore the spatial spread and front propagation
dynamics of \eqref{main-eq} in the case that
  $f(t,x,u)$ is    periodic in
$t$  and $x$  and  satisfies proper  monostablility
 assumptions. More precisely,  let (H0) stands the following assumption.

 \medskip
 \noindent {\bf (H0)} {\it $f(t,x,u)$ is $C^1$ in $(t,x,u)\in \RR\times\RR^N\times \RR^+$, $f(t,x,u)=f(t,x,0)$ for $t\in\RR$, $x\in\RR^N$, and $u\le 0$,
  and  $f(\cdot+T,\cdot,\cdot)=f(\cdot,\cdot+p_i{\bf e_i},\cdot)=f(\cdot,\cdot,\cdot)$,
${\bf e_i}=(\delta_{i1},\delta_{i2}, \cdots,\delta_{iN})$, $\delta_{ij}=1$ if $i=j$ and $0$ if $i\not =j$,
$i,j=1,2,\cdots,N$.}
 \medskip

 Throughout the rest of this paper, we assume $f$ satisfies (H0).
 Let
\begin{equation}
\label{x-pp-space} \mathcal{X}_p=\{u\in C(\RR\times \RR^N,\RR)|u(\cdot+T,\cdot)=u(\cdot,\cdot+p_i{\bf e_i})=u(\cdot,\cdot),\quad i=1,\cdots,N\}
\end{equation}
with norm $\|u\|_{\mathcal{X}_p}=\sup_{(t,x)\in\RR\times \RR^N}|u(t,x)|$, and
\begin{equation}
\label{x-pp-positive-space} \mathcal{X}_p^+=\{u\in \mathcal{X}_p\,|\, u(t,x)\geq 0\quad \forall (t,x)\in\RR\times \RR^N\}.
\end{equation}
Let  $I$ be the identity map on $\mathcal{X}_p$, and $\mathcal{K}$, $a_0(\cdot,\cdot)I:\mathcal{X}_p\to \mathcal{X}_p$ be defined by
\begin{equation}
\label{k-delta-op} \big(\mathcal{K} u\big)(t,x)=\int_{\RR^N}k(y-x)u(t,y)dy,
\end{equation}
\begin{equation}
\label{a-op} (a_0(\cdot,\cdot)Iu)(t,x)=a_0(t,x)u(t,x),
\end{equation}
where $k(\cdot)$ is as in \eqref{main-eq} and  $a_0(t,x)=f(t,x,0)$.
Let $\sigma(-\p_t +\mathcal{K}-I+a_0(\cdot,\cdot)I)$ be the spectrum of $-\p_t+\mathcal{K}-I+a_0(\cdot,\cdot)I$ acting on $\mathcal{X}_p$.
 The monostablility assumptions are then stated as follows:

\medskip
\medskip
 \noindent{\bf (H1)} {\it
 $\frac{\p f(t,x,u)}{\p u}<0$ for $t\in\RR$,  $x\in\RR^N$ and $u\in\RR^+$ and
 $f(t,x,u)<0$ for $t\in\RR$,  $x\in\RR^N$ and $u\gg 1$.}

 \medskip
\noindent{\bf (H2)} {\it $u\equiv 0$ is linearly unstable in $X_p$, that is,  $\lambda_0(a_0)>0$,
where $\lambda_0(a_0):=\sup\{{\rm Re}\lambda\,|\, \lambda\in\sigma(-\p_t+\mathcal{K}-I+a_0(\cdot,\cdot)I)$.}

 \medskip

It is proved in \cite{RaSh} that   (H1) and (H2) imply  that  \eqref{main-eq} has exactly two time periodic
 solutions in $\mathcal{X}_p^+$,  $u= 0$ and $u= u^*(t,x)$, and $u= 0$ is linearly unstable and $u=u^*(t,x)$
  is
asymptotically stable with respect to positive perturbations  in $X^+_p$ (see \cite{RaSh} for details), where
\begin{equation}
\label{x-p-space}
X_p=\{u\in C(\RR^N,\RR)\,|\, u(\cdot+p {\bf e_i})=u(\cdot)\}
\end{equation}
with maximum norm
and
\begin{equation}
\label{x-p-positive-space}
X_p^+=\{u\in X_p\,|\, u(x)\ge 0\,\,\forall \,\, x\in\RR^N\}.
\end{equation}
Hence (H1) and (H2)  are called monostability assumptions.

In the current paper, we investigate the spreading feature and traveling wave solutions of \eqref{main-eq}.
Let
\begin{equation}
\label{x-space}
X=\{u\in C(\RR^N,\RR)\,|\, u\,\,\,\text{is uniformly continuous and bounded}\}
\end{equation}
with supremum norm and
\begin{equation}
\label{x-positive-space}
X^+=\{u\in X\,|\, u(x)\ge 0\quad \forall\,\, x\in\RR^N\}.
\end{equation}
By general semigroup theory, for any $u_0\in X$, \eqref{main-eq} has a unique solution $u(t,x;u_0)$ with $u(0,x;u_0)=u_0(x)$.
By comparison principle, if $u_0\in X^+$, then $u(t,\cdot;u_0)$ exists for all $t\ge 0$ and
$u(t,\cdot;u_0)\in X^+$ (see  Proposition \ref{comparison-nonlinear-prop} for details).

 For given $\xi\in S^{N-1}$ and $\mu\in\RR$, let $\lambda_0(\xi,\mu,a_0)$ be the principal
spectrum point of the eigenvalue problem
\begin{equation}
\label{eigenvalue-eq0}
\begin{cases}
-u_t+\int_{\RR^N}e^{-\mu(y-x)\cdot\xi}k(y-x)u(t,y)dy-u(t,x)+a_0(t,x)u(t,x)=\lambda u(t,x)\cr
u(\cdot,\cdot)\in\mathcal{X}_p
\end{cases}
\end{equation}
(see Definition \ref{principal-spectrum-point-def}  for details).
Let $X^+(\xi)$ be defined by
\begin{equation}
\label{X-space-in-direction-xi}
X^+(\xi)=\{u\in X^+\,|\, \inf_{x\cdot\xi\ll -1}u(x)>0,\,\, \sup_{x\cdot\xi\gg 1}u(x)=0\}.
\end{equation}
Roughly, a real number $c^*(\xi)\in\RR$ is called the {\it spreading speed} of \eqref{main-eq} in the direction of $\xi\in S^{N-1}$ if for
any $u_0\in X^+(\xi)$,
\begin{equation*}
\limsup_{t\to\infty}\sup_{x\cdot\xi \leq ct}|u(t,x;u_0)-u^*(t,x)|= 0 \quad \forall\,\, c<c^*(\xi)
\end{equation*}
and
\begin{equation*}
  \limsup_{t\to\infty}\sup_{x\cdot\xi \geq ct}u(t,x;u_0)=0\quad \forall \,\, c>c^*(\xi)
\end{equation*}
(see Definition \ref{spreading-speed-interval-def} for details).
Among others, we prove

\medskip

\noindent $\bullet$ $c^*(\xi):=\inf_{\mu>0}\frac{\lambda_0(\xi,\mu,a_0)}{\mu}$ is the spreading speed of \eqref{main-eq}
in the direction of $\xi$ (see Theorem \ref{spreading-thm1} for details). Moreover, the spreading speed $c^*(\xi)$ is of some
important spreading features (see Theorem \ref{spreading-thm2} for details).

\medskip

\noindent $\bullet$ If $\lambda_0(\xi,\mu,a_0)$ is the principal eigenvalue of \eqref{eigenvalue-eq0} for all $\mu>0$
(see Definition \ref{principal-spectrum-point-def} for the definition of principal eigenvalue), then for any $c>c^*(\xi)$,
\eqref{main-eq} has a  continuous (periodic) traveling wave solution $u(t,x)=\Phi(x-ct,t,ct)$ connecting $u^*$ and $0$ in the direction of $\xi$
(i.e. $\Phi(x,t,z)$ is continuous in $x,t$, and $z$, is  periodic in $t$ and $z$, and $\Phi(x,t,z)-u^*(t,x+z)\to 0$ as $x\cdot\xi\to -\infty$ and
$\Phi(x,t,z)\to 0$ as $x\cdot\xi\to \infty$) (see Theorem   \ref{existence-thm} for details).

\medskip

To prove these results, we first establish
some  new results on the principal eigenvalue of nonlocal dispersal operators with time periodic dependence, and among those, we prove

\medskip

\noindent $\bullet$ If $\lambda_0(\xi,\mu,a_0)$ is the principal eigenvalue of \eqref{eigenvalue-eq0}   for all $\mu>0$, then
$\lambda_0(\xi,\mu,a_0)$ is algebraically simple,  and
$\lambda_0(\xi,\mu,a_0)$ and $\phi(\cdot,\cdot;\xi,\mu)$ are smooth in  $\mu$, where $\phi(\cdot,\cdot;\xi,\mu)$ is the positive eigenfunction of
\eqref{eigenvalue-eq0} associated to $\lambda_0(\xi,\mu,a_0)$ with $\|\phi(\cdot,\cdot;\xi,\mu)\|=1$.
 (see Theorem \ref{PE-simplicity-thm} for details).

\medskip

It should be pointed out that the above property of principal eigenvalue of nonlocal dispersal operators is of independent interest.
The first two authors of the current paper developed  in \cite{RaSh} some criteria for the existence of principal eigenvalue of nonlocal dispersal operators.
The principal
eigenvalue theory for nonlocal dispersal operators  established   in \cite{RaSh} and the new principal eigenvalue theory for nonlocal dispersal operators
developed  in the current paper
 will play an important role in the proofs of the existence of spreading speeds and traveling wave solutions of \eqref{main-eq}.

It should also be pointed out that the existence of spreading speeds of \eqref{main-eq} does not require the existence of principal eigenvalue of \eqref{eigenvalue-eq0}.
It follows from \cite[Theorem B]{RaSh} that if $a_0(t,x)$ is $C^N$ and $1\le N\le 2$, then $\lambda_0(\xi,\mu,a_0)$ is the principal  eigenvalue of \eqref{eigenvalue-eq0}   for all $\mu>0$.
Hence if $1\le N\le 2$ and $a_0(t,x)$ is $C^N$, then
for any $c>c^*(\xi)$,
\eqref{main-eq} has a  continuous (periodic) traveling wave solution $u(t,x)=\Phi(x-ct,t,ct)$ connecting $u^*$ and $0$ in the direction of $\xi$.
When $N\ge 3$, \eqref{eigenvalue-eq0} may not have a principal eigenvalue (see \cite{ShZh0} for an example).
If $\lambda_0(\xi,\mu,a_0)$ is not a principal eigenvalue of \eqref{eigenvalue-eq0}  for some $\mu>0$,
it remains open whether \eqref{main-eq} has a traveling wave solution connecting $u^*(\cdot,\cdot)$ and $0$ in the direction of $\xi$ with any speed
$c>c^*(\xi)$ (this remains open even when $f(t,x,u)\equiv f(x,u)$ is time independent but space periodic).

The results of the current paper extend the existence of spreading speed $c^*(\xi)$ and its spreading properties in \cite{ShZh0} and
the existence of traveling waves with  speed $c>c^*(\xi)$
in \cite{CoDaMa2} and
\cite{ShZh1} for  spatially periodic case to  both space and time periodic case. In the case
that $f(t,x,u)=f(x,u)$ is spatially periodic, the existence of traveling waves with speed $c=c^*(\xi)$ is also proved in \cite{CoDaMa2} and
the uniqueness and stability of traveling waves with speed $c>c^*(\xi)$ are proved in \cite{ShZh1}. The existence of traveling waves with speed $c=c^*(\xi)$ and
uniqueness and stability of traveling waves in the case that $f$ is both space and time periodic remain open.

 The rest of the paper is organized as follows. In section 2, we  present some comparison principle for nonlocal evolution equations. We establish some new
 principal eigenvalue theory for nonlocal dispersal operators with time dependence in section 3. Spatial spreading speeds and traveling wave solutions of
 \eqref{main-eq} are investigated in sections 4 and 5, respectively.

\medskip

\noindent {\bf Acknowledgment.} The authors thank the referees for the careful reading of the manuscript
and valuable comments and suggestions, which improved the presentation considerably.

\section{Comparison Principle for Nonlocal Dispersal Equations}

In this section, we present comparison principles for solutions of nonlocal dispersal equations.

First,  consider  the following nonlocal linear evolution equation,
\begin{equation}
\label{linear-new-eq1} \frac{\p u}{\p t}=\int_{\RR^N}e^{-\mu
(y-x)\cdot\xi}k(y-x)u(t,y)dy-u(t,x)+a(t,x)u(t,x),\quad x\in\RR^N
\end{equation}
where $\mu\in\RR$, $\xi\in S^{N-1}$, and $a(t,\cdot)\in X_p$ and $a(t+T,x)=a(t,x)$. Note that if $\mu=0$ and $a(t,x)=a_0(t,x)(:=f(t,x,0))$, \eqref{linear-new-eq1}
is  the linearization of \eqref{main-eq} at $u\equiv 0$.

Throughout this section, we assume that $\xi\in S^{N-1}$ and $\mu\in\RR$ are fixed, unless otherwise specified.

Let $X_p$ and $X$ be as in \eqref{x-p-space} and \eqref{x-space}, respectively.
It follows from the general linear semigroup theory (see \cite{Hen} or \cite{Paz}) that  for every $u_0\in
X$, \eqref{linear-new-eq1} has a unique solution $u(t,\cdot;u_0,\xi,\mu,a) \in X$  with
$u(0,x;u_0,\xi,\mu,a)=u_0(x)$. Put
\begin{equation}
\label{phi-xi-mu-eq} \Phi(t;\xi,\mu,a)u_0=u(t,\cdot;u_0,\xi,\mu,a).
\end{equation}
Note  that if $u_0\in X_p$, then $\Phi(t;\xi,\mu,a)u_0\in X_p$ for $t\geq 0$.

Let $X^+_p$ and $X^+$ be as in \eqref{x-p-positive-space} and \eqref{x-positive-space}, respectively. Let
\begin{equation}
\label{x-p-interior-space} {\rm Int}(X_p^+)=\{v\in X_p|v(x)>0,x\in\RR^N\}.
\end{equation}
For $v_1,v_2\in X_p$, we define
$$
v_1\leq v_2\quad (v_1\geq v_2)\quad {\rm if }\quad v_2-v_1\in X_p^+\quad (v_1-v_2\in X_p^+),
$$
and
$$
v_1\ll v_2\quad (v_1\gg v_2)\quad {\rm if}\quad v_2-v_1\in {\rm Int}(X_p^+)\quad (v_1-v_2\in {\rm Int}(X_p^+)).
$$
For $u_1,u_2\in X$, we define
$$u_1\leq u_2\quad (u_1\geq u_2)\quad {\rm if}\quad u_2-u_1\in X^+\quad (u_1-u_2\in X^+).
$$

A continuous function $u(t,x)$ on $[0,T)\times \RR^N$ is called a {\it super-solution} or {\it sub-solution} of
\eqref{linear-new-eq1} if $\frac{\p u}{\p t}$ exists and is continuous on $[0,T)\times\RR^N$ and satisfies
$$
\frac{\p u}{\p t}\geq \int_{\RR^N} e^{-\mu(y-x)\cdot\xi}k(y-x)u(t,y)dy-u(t,x)+a(t,x)u(t,x),\quad x\in\RR^N
$$
or
$$
\frac{\p u}{\p t}\leq \int_{\RR^N} e^{-\mu(y-x)\cdot\xi)}k(y-x)u(t,y)dy-u(t,x)+a(t,x)u(t,x),\quad x\in\RR^N
$$
for $t\in [0,T)$.

\begin{proposition}[Comparison principle for linear equations]
\label{comparison-linear-prop} $\quad$
\begin{itemize}
\item[(1)]
If $u_1(t,x)$ and $u_2(t,x)$ are sub-solution and super-solution of \eqref{linear-new-eq1} on $[0,T)$,
respectively, $u_1(0,\cdot)\leq u_2(0,\cdot)$,  and $u_2(t,x)-u_1(t,x)\geq -\beta_0$ for $(t,x)\in [0,T)\times
\RR^N$ and some $\beta_0>0$, then
$u_1(t,\cdot)\leq u_2(t,\cdot)\quad {\rm for}\quad t\in [0,T).$

\item[(2)] Suppose that $u_1,u_2\in X_p$ and $u_1\leq u_2$, $u_1\not =u_2$.
Then  $\Phi(t;\xi,\mu,a)u_1\ll
\Phi(t;\xi,\mu,a)u_2$ for all $t>0$.
\end{itemize}
\end{proposition}

\begin{proof}
(1) If follows from the arguments in \cite[Proposition 2.1]{ShZh0}.

(2) It follows from the arguments in \cite[Proposition 2.2]{ShZh0}.
\end{proof}

 For given $\rho\geq 0$, let
\begin{equation}
\label{x-mu-space} X(\rho)=\{u\in C(\RR^N,\RR)\,|\,\text{the function}\,\, x\mapsto e^{-\rho \|x\|}u(x)\,\,
\text{belongs to}\, \,X\}
\end{equation}
equipped with the norm $\|u\|_{X(\rho)}=\sup_{x\in\RR^N}e^{-\rho\|x\|}|u(x)|$.

\medskip

\begin{remark}
\label{weighted-space-rk}
  For every $u_0\in
X(\rho)$, \eqref{linear-new-eq1} has a unique solution $u(t,\cdot;u_0,\xi,\mu,a) \in X(\rho)$  with
$u(0,x;u_0,\xi,\mu,a)=u_0(x)$.
Moreover, for any $u_1,u_2\in X(\rho)$ with $u_1\le u_2$ and $u(t,x;u_2,\xi,\mu,a)- u(t,x;u_1,\xi,\mu,a)\ge -\beta$ for some
$\beta>0$ and any $(t,x)\in [0,\infty)\times\RR^N$, $u(t,x;u_1,\xi,\mu,a)\le u(t,x;u_2,\xi,\mu,a)$ for $t\ge 0$ and $x\in\RR^N$.
\end{remark}

\medskip

Next, consider \eqref{main-eq}.
By general nonlinear semigroup theory (see \cite{Hen} or \cite{Paz}), \eqref{main-eq} has a unique (local)
solution $u(t,x;u_0)$ with $u(0,x;u_0)=u_0(x)$ for every $u_0\in X$. Also if $u_0\in X_p$, then $u(t,x;u_0)\in
X_p$ for $t$ in the existence interval of the solution $u(t,x;u_0)$.

A continuous function $u(t,x)$ on $[0,T)\times \RR^N$ is called a {\it super-solution} or {\it sub-solution} of
\eqref{main-eq} if $\frac{\p u}{\p t}$ exists and is continuous on $[0,T)\times\RR^N$ and satisfies
$$
\frac{\p u}{\p t}\geq \int_{\RR^N} e^{-\mu(y-x)\cdot\xi}k(y-x)u(t,y)dy-u(t,x)+u(t,x)f(t,x,u(t,x)),\quad x\in\RR^N
$$
or
$$
\frac{\p u}{\p t}\leq \int_{\RR^N} e^{-\mu(y-x)\cdot\xi}k(y-x)u(t,y)dy-u(t,x)+u(t,x)f(t,x,u(t,x)),\quad x\in\RR^N
$$
for $t\in [0,T)$.

\begin{proposition}[Comparison principle for nonlinear equations]
\label{comparison-nonlinear-prop} $\quad$
\begin{itemize}
\item[(1)] If $u_1(t,x)$ and $u_2(t,x)$ are bounded sub- and super-solutions of \eqref{main-eq} on $[0,T)$,
respectively, and
$u_1(0,\cdot)\leq u_2(0,\cdot)$, then $u_1(t,\cdot)\leq u_2(t,\cdot)$ for $t\in[0,T)$.

\item[(2)] If $u_1,u_2\in X_p$ with $u_1\le u_2$ and $u_1\not = u_2$, then $u(t,\cdot;u_1)\ll u(t,\cdot;u_2)$ for every $t>0$ at which both $u(t,\cdot;u_1)$ and
$u(t,\cdot;u_2)$ exist.

\item[(3)] For every $u_0\in X^+$, $u(t,x;u_0)$ exists for all $t\geq 0$.
\end{itemize}
\end{proposition}

\begin{proof}
If follows from the arguments in \cite[Proposition 2.1]{ShZh0}
and the arguments in \cite[Proposition 2.2]{ShZh0}.
\end{proof}

\begin{remark}
\label{general-space-rk}
Let
$$
\tilde X=\{u:\RR^N\to \RR\,|\, u\,\, \text{is Lebesgue measurable and bounded}\}
$$
equipped with the norm $\|u\|=\sup_{x\in\RR^N}|u(x)|$, and
$$
\tilde X^+=\{u\in\tilde X\,|\, u(x)\ge 0\,\, \forall\,\, x\in\RR^N\}.
$$
 By general semigroup theory, for any $u_0\in X$, \eqref{main-eq} has also
a unique (local) solution $u(t,\cdot;u_0)\in \tilde X$ with $u(0,x;u_0)=u_0(x)$. Similarly, we can define measurable sub- and super-solutions of \eqref{main-eq}.
Proposition \ref{comparison-nonlinear-prop} (1) and (3) also hold for bounded measurable sub-, super-solutions and solutions.
\end{remark}

Observe that $\tilde X$ is different from $L^\infty(\RR^N)$. For given $u,v\in\tilde X$, $u=v$ in $\tilde X$ indicates that $u(x)=v(x)$ for all
$x\in\RR^N$, while  $u=v$ in $L^\infty(\RR^N)$ indicates that $u(x)=v(x)$ for a. e. $x\in\RR^N$.

 \section{Principal Spectrum Points and Principal Eigenvalues of Nonlocal Dispersal Operators}

In this section, we present some principal spectrum point and principal eigenvalue theory for
time periodic nonlocal dispersal operators. Throughout this section, $r(A)$ denotes the spectral radius of an operator
$A$ on some Banach space.

Let $\mathcal{X}_p$ be as in \eqref{x-pp-space}.
Consider the following eigenvalue problem
\vspace{-.05in}\begin{equation}
\label{eigenvalue-eq}-v_t+ \big( \mathcal{K}_{\xi,\mu} -I +a(\cdot,\cdot) I
\big)v=\lambda v,\quad v\in \mathcal{X}_p,
\vspace{-.05in}\end{equation}
where $\xi\in S^{N-1}$, $\mu\in\RR$, and $a(\cdot,\cdot)\in \mathcal{X}_p$. The operator  $a(\cdot,\cdot)I$  has the same
 meaning
  as in \eqref{a-op} with $a_0(\cdot,\cdot)$ being replaced
 by $a(\cdot,\cdot)$,
 and $\mathcal{K}_{\xi,\mu}:\mathcal{X}_p\to\mathcal{X}_p$ is defined by
\vspace{-.05in}\begin{equation}
\label{k-delta-xi-mu-op}
(\mathcal{K}_{\xi,\mu}v)(t,x)=\int_{\RR^N}e^{-\mu(y-x)\cdot\xi}k(y-x)v(t,y)dy.
\vspace{-.05in}\end{equation}
We point out the following relation between \eqref{main-eq} and
\eqref{eigenvalue-eq}: if $u(t,x)=e^{-\mu(
x\cdot\xi-\frac{\lambda}{\mu}t)}\phi(t,x)$ with $\phi\in
\mathcal{X}_p\setminus\{0\}$ is a solution of
 the  linearization of \eqref{main-eq} at $u= 0$,
\vspace{-.05in}\begin{equation}
\label{linearization-eq0} \frac{\p u}{\p t}=\int_{\RR^N}
k(y-x)u(t,y)dy-u(t,x)+a_0(t,x)u(t,x),\quad x\in\RR^N,
\vspace{-.05in}\end{equation}
where $a_0(t,x)=f(t,x,0)$,  then $\lambda$ is an eigenvalue of
\eqref{eigenvalue-eq}  with $a(t,\cdot)=a_0(t,\cdot)$ or
$-\p_t+\mathcal{K}_{\xi,\mu}-I+a_0(\cdot,\cdot)I$  and $v=\phi(t,x)$ is a
corresponding eigenfunction.

 Let
$\sigma(-\p_t+ \mathcal{K}_{\xi,\mu}- I+a(\cdot,\cdot)I)$ be the spectrum of
$-\p_t+ \mathcal{K}_{\xi,\mu}- I+a(\cdot,\cdot)I$ on $\mathcal{X}_p$. Let
\vspace{-.05in}\begin{equation*}
\lambda_0(\xi,\mu,a):=\sup\{{\rm Re}\lambda\,|\,\lambda\in \sigma( -\p_t+\mathcal{K}_{\xi,\mu}- I+a(\cdot,\cdot)I)\}.
\vspace{-.05in}\end{equation*}
Observe that if $\mu=0$, \eqref{eigenvalue-eq} is independent of $\xi$
and hence we put
\begin{equation}
\label{lambda-delta-a}
\lambda_0(a):=\lambda_0(\xi,0,a)\quad \forall\,\, \xi\in
S^{N-1}.
\end{equation}

\begin{definition}
\label{principal-spectrum-point-def}
We call $\lambda_0(\xi,\mu,a)$  the {\rm  principal spectrum point} of $-\p_t+\mathcal{K}_{\xi,\mu}- I+a(\cdot,\cdot)I$.
$\lambda_0(\xi,\mu,a)$ is called the {\rm principal eigenvalue} of $
-\p_t+\mathcal{K}_{\xi,\mu}- I+a(\cdot,\cdot)I$ or $ -\p_t+\mathcal{K}_{\xi,\mu}- I+a(\cdot,\cdot)I$ is said to {\rm have
a principal eigenvalue} if  $\lambda_0(\xi,\mu,a)$ is an isolated
 eigenvalue of $-\p_t+\mathcal{K}_{\xi,\mu}-I+a(\cdot, \cdot)I$ with  finite algebraic multiplicity and a positive
  eigenfunction
 $v\in \mathcal{X}_p^+$,  and for every $\lambda\in
\sigma(-\p_t+ \mathcal{K}_{\xi,\mu}- I+a(\cdot)I) \setminus
\{\lambda_0(\xi,\mu,a)\}$, ${\rm
Re}\lambda\le \lambda_0(\xi,\mu,a)$.
\end{definition}

Observe that $-\p_t + \mathcal{K}_{\xi,\mu}- I+a(\cdot,\cdot)I$ may not have a principal eigenvalue  (see an example in
\cite{ShZh0}), which reveals some essential difference between random dispersal operators and nonlocal dispersal operators.
Let
$$
\hat a(x)=\frac{1}{T}\int_0^T a(t,x)dt.
$$

The following proposition provides necessary and sufficient condition for $-\p_t+\mathcal{K}_{\xi,\mu}-I+a(\cdot,\cdot)I$
to have a principal eigenvalue.

\begin{proposition}
\label{PE-necessary-sufficient-prop}
$\lambda_0(\xi,\mu,a)$ is the principal eigenvalue of $-\p_t+\mathcal{K}_{\xi,\mu}-I+a(\cdot,\cdot)I$ if and only if
$\lambda_0(\xi,\mu,a)>-1+\max_{x\in\RR^N}\hat a(x)$.
\end{proposition}

\begin{proof}
It follows from \cite[Theorem A]{RaSh}.
\end{proof}

The following proposition provides a very useful sufficient condition for
$\lambda_0(\xi,\mu,a)$ to be the principal eigenvalue of  $-\p_t+\mathcal{K}_{\xi,\mu}-I+a(\cdot,\cdot)I$.

\begin{proposition}
\label{PE-sufficient-prop}
 If $a(t,\cdot)$ is $C^N$ and  the partial
derivatives of $\hat a(x)$ up to order $N-1$ at some $x_0$ are zero {\rm (we refer this to as a vanishing condition)}, where $x_0$ is such that
$\hat a(x_0)=\max_{x\in\RR^N}\hat a(x)$, then $\lambda_0(\xi,\mu,a)$ is the principal eigenvalue of $-\p_t+\mathcal{K}_{\xi,\mu}-I+a(\cdot,\cdot)I$
for all $\xi\in S^{N-1}$ and $\mu\in\RR$.
\end{proposition}

\begin{proof}
It follows from the arguments of \cite[Theorem B (1)]{RaSh}.
\end{proof}

\begin{proposition}
\label{PE-finite-multiplicity-prop}
Each $\lambda\in\sigma(-\p_t+\mathcal{K}_{\xi,\mu}-I+a(\cdot,\cdot)I)$ with ${\rm Re}\lambda>-1+\max_{x\in\RR^N}\hat a(x)$ is an isolated eigenvalue
with finite algebraic multiplicity.
\end{proposition}

\begin{proof}
It follows from \cite[Proposition 2.1(ii)]{Bur}.
\end{proof}

The following theorem shows that  the  principal eigenvalue of $-\p_t+\mathcal{K}_{\xi,\mu}-I+a(\cdot,\cdot)I$ (if it exists) is algebraically simple,
which is new and plays an important role in the proof of the existence of spreading speeds of \eqref{main-eq}.

\begin{theorem}
\label{PE-simplicity-thm}
Suppose that $\lambda_0(\xi_0,\mu_0,a)$ is the principal eigenvalue of $-\p_t+\mathcal{K}_{\xi_0,\mu_0}-I+a(\cdot,\cdot)I$.
Then $\lambda_0(\xi,\mu,a)$ is an algebraically simple principal eigenvalue of $-\p_t+\mathcal{K}_{\xi,\mu}-I+a(\cdot,\cdot)I$ with
 a positive eigenfunction $\phi(\cdot,\cdot;\xi,\mu)$, $\|\phi(\cdot,\cdot;\xi,\mu)\|=1$,  and
$\lambda_0(\xi,\mu,a)$ and $\phi(\cdot,\cdot;\xi,\mu)$ are smooth in $\xi$ and $\mu$ for $(\xi,\mu)$ near $(\xi_0,\mu_0)$.
\end{theorem}

\begin{proof}
First of all, note that for $\alpha>-1+\max_{x\in\RR^N}\hat a(x)$,  $(\alpha I+\p_t +I-aI)^{-1}$ exists (see \cite[Proposition 3.5]{RaSh}).  For given $\alpha>-1+\max_{x\in\RR^N}\hat a(x)$,
let
$$
(U_{\alpha,\xi,\mu}u)(t,x)=\int_{\RR^N}e^{-\mu (y-x)\cdot\xi}k(y-x)(\alpha+\p_t+I-aI)^{-1}u(t,y)dy
$$
and
$$
r(\alpha)=r(U_{\alpha,\xi,\mu}).
$$
By \cite[Proposition 3.6]{RaSh}, $U_{\alpha,\xi,\mu}:\mathcal{X}_p\to\mathcal{X}_p$ is a positive and compact operator.

Next, suppose that $\lambda_0(\xi_0,\mu_0,a)$ is the principal eigenvalue of $-\p_t+\mathcal{K}_{\xi_0,\mu_0}-I+a(\cdot,\cdot)I$.
By \cite[Proposition 3.9]{RaSh}, $\lambda_0(\xi_0,\mu_0,a)$ is an isolated geometrically simple eigenvalue  of $-\p_t+\mathcal{K}_{\xi_0,\mu_0}-I+a(\cdot,\cdot)I$.
Let $\alpha_0=\lambda_0(\xi_0,\mu_0,a)$. This implies  that
$r(\alpha_0)=1$ and $r(\alpha_0)$ is an isolated geometrically simple eigenvalue of
$U_{\alpha_0,\xi_0,\mu_0}$ with $\phi(\cdot,\cdot;\xi_0,\mu_0)$ being a positive eigenfunction. We claim that $r(\alpha_0)$ is an algebraically simple isolated eigenvalue
of $U_{\alpha_0,\xi_0,\mu_0}$ with a positive eigenfunction $\phi(\cdot,\cdot)$, or equivalently, $(I-U_{\alpha_0,\xi_0,\mu_0})^2\psi=0$  ($\psi\in \mathcal{X}_p$)
iff $\psi\in {\rm span}\{\phi\}$. If fact, suppose that $\psi\in\mathcal{X}_p\setminus\{0\}$ is such that $(I-U_{\alpha_0,\xi_0,\mu_0})^2\psi=0$. Then
\begin{equation}
\label{simplicity-eq1}
(I-U_{\alpha_0,\xi_0,\mu_0})\psi=\gamma \phi
\end{equation}
 for some
$\gamma\in\RR$. We prove that $\gamma=0$. Assume that $\gamma\not =0$. Without loss of generality, we assume that $\gamma>0$.
By \eqref{simplicity-eq1} and $U_{\alpha_0,\xi_0,\mu_0}\phi=\phi$, we have
\begin{equation}
\label{simplicity-eq2}
\psi=U_{\alpha_0,\xi_0,\mu_0}\psi+\gamma\phi=U_{\alpha_0,\xi_0,\mu_0}(\psi+\gamma\phi).
\end{equation}
Then by \eqref{simplicity-eq2} and $U_{\alpha_0,\xi_0,\mu_0}\phi=\phi$,
\begin{align*}
\psi+\gamma\phi&=U_{\alpha_0,\xi_0,\mu_0}(\psi+\gamma\phi)+\gamma\phi\\
&=U_{\alpha_0,\xi_0,\mu_0}(\psi+2\gamma\phi)
\end{align*}
and hence
\begin{align*}
\psi&=U_{\alpha_0,\xi_0,\mu_0}(\psi+\gamma\phi)\\
&=U_{\alpha_0,\xi_0,\mu_0}\Big(U_{\alpha_0,\xi_0,\mu_0}(\psi+2\gamma\phi)\Big)\\
&=U^2_{\alpha_0,\xi_0,\mu_0}(\psi+2\gamma \phi).
\end{align*}
By induction, we have
$$\psi=U^n_{\alpha_0,\xi_0,\mu_0}(\psi+n\gamma\phi)\quad \forall \,\, n\ge 1.
$$
This implies that
$$
\frac{\psi}{n}=U^n_{\alpha_0,\xi_0,\mu_0}(\frac{\psi}{n}+\gamma\phi).
$$
Note that $\phi(t,x)>0$ and then
$$
\frac{\psi(t,x)}{n}+\gamma \phi(t,x)>0\quad \forall\, \, n\gg 1.
$$
By the positivity of $U_{\alpha_0,\xi_0,\mu_0}$, we then have
$$
\frac{\psi(t,x)}{n}>0\quad \forall\,\, n\gg 1
$$
and then
$$
\frac{\psi(t,x)}{n}-\gamma\phi(t,x)=\big(U^n_{\alpha_0,\xi_0,\mu_0}(\frac{\psi}{n})\big)(t,x)>0\quad \forall\,\, n\gg 1.
$$
It then follows that
$$
-\gamma \phi(t,x)\ge 0
$$
and then
$$
\gamma\le 0.
$$
This is a contradiction. Therefore, $\gamma=0$ and then by \eqref{simplicity-eq1},
$$
\psi\in{\rm span}\{\phi\}.
$$
The claim is thus proved.

Now, we prove that $\lambda_0(\xi_0,\mu_0,a)$ is an algebraically simple eigenvalue  of $-\p_t+\mathcal{K}_{\xi_0,\mu_0}-I+a(\cdot,\cdot)I$ or
equivalently, $(-\p_t+\mathcal{K}_{\xi_0,\mu_0}-I+a(\cdot,\cdot)I-\alpha_0I )^2\psi=0$ iff $\psi\in {\rm span}\{\phi\}$. By the above arguments,
 there are one dimensional subspace $\mathcal{X}_{1,p}={\rm span}(\phi)$ and one-codimensional subspace $\mathcal{X}_{2,p}$ of
$\mathcal{X}_p$ such that
$$
\mathcal{X}_p=\mathcal{X}_{1,p}\oplus \mathcal{X}_{2,p},
$$
\begin{equation}
\label{simplicity-eq3}
U_{\alpha_0,\xi_0,\mu_0}\mathcal{X}_{1,p}=\mathcal{X}_{1,p},\quad U_{\alpha_0,\xi_0,\mu_0}\mathcal{X}_{2,p}\subset \mathcal{X}_{2,p},
\end{equation}
and
$$
1\not\in \sigma(U_{\alpha_0,\xi_0,\mu_0}|_{\mathcal{X}_{2,p}}).
$$
Suppose that $\psi\in\mathcal{X}_p$ is such that
$$
(-\p_t+\mathcal{K}_{\xi_0,\mu_0}-I+a(\cdot,\cdot)I-\alpha_0I )^2\psi=0.
$$
Then there is $\gamma\in\RR$ such that
\begin{equation}
\label{simplicity-eq4}
(-\p_t+\mathcal{K}_{\xi_0,\mu_0}-I+a(\cdot,\cdot)I-\alpha_0I )\psi=\gamma \phi.
\end{equation}
Let $\psi_i\in\mathcal{X}_{i,p}$ ($i=1,2$) be such that
$$
\psi=(\alpha_0 I+\p_t+I-aI)^{-1} \psi_1+ (\alpha_0 I+\p_t+I-aI)^{-1}\psi_2.
$$
Then
\begin{align*}
(-\p_t+\mathcal{K}_{\xi_0,\mu_0}-I+a(\cdot,\cdot)I-\alpha_0I )\psi&=(U_{\alpha_0,\xi_0,\mu_0}-I)\psi_1+(U_{\alpha_0,\xi_0,\mu_0}-I)\psi_2\\
&=(U_{\alpha_0,\xi_0,\mu_0}-I)\psi_2\\
&=\gamma\phi.
\end{align*}
This  together with \eqref{simplicity-eq3} implies that $\gamma\phi\in\mathcal{X}_{2,p}$ and hence $\gamma=0$.
By \eqref{simplicity-eq4}, $\psi\in {\rm span}\{\phi\}$ and hence $\lambda_0(\xi_0,\mu_0,a)$ is an algebraically simple eigenvalue  of $-\p_t+\mathcal{K}_{\xi_0,\mu_0}-I+a(\cdot,\cdot)I$.
The rest of the proposition follows from perturbation theory of isolated eigenvalues of closed operators.
\end{proof}

\begin{proposition}
 \label{PE-minimum-prop} For given $\xi\in S^{N-1}$,
 assume $\lambda_0(\xi,0,a)>0$ and $\lambda_0(\xi,\mu,a)$ is the principal eigenvalue
 of $-\p_t+\mathcal{K}_{\xi,\mu}-I+a(\cdot,\cdot)I$ for $\mu>0$.
  There is $\mu^*(\xi)\in (0,\infty)$ such that
 \begin{equation}
 \label{mu-star-eq}
 \frac{\lambda_0(\xi,\mu^*(\xi),a)}{\mu^*(\xi)}=\inf_{\mu>0}\frac{\lambda_0(\xi,\mu,a)}{\mu}.
 \end{equation}
 \end{proposition}

\begin{proof}
Note that $\lambda_0(\xi,\mu,a)\geq\lambda_0(\xi,\mu,a_{\min})$,
 and
$$
\lambda_0(\xi,\mu,a_{\min})=\int_{\RR^N}e^{-\mu y\cdot \xi}k(y)dy-1+a_{\min}
$$ with 1 as an eigenfunction.
Note also that there is $k_0>0$ such that $k(y)\ge k_0$ for $\|y\|\le \frac{r_0}{2}$.
Let $m_{n}(\xi)=k_0\int_{ \|y\|\le \frac{r_0}{2}}\frac{(- y\cdot \xi)^{n}}{n!}dy$. Then, for $\mu>0$
\begin{align*}
\int_{\RR^N}e^{-\mu y\cdot \xi}k(y)dy-1+a_{\min}&\ge k_0 \int_{\|y\|\le \frac{r_0}{2}} e^{-\mu y\cdot\xi}dy -1+a_{\min}\\
&= k_0\sum_{n=0}^{\infty}\int_{\|y\|\le \frac{r_0}{2}}\frac{(-\mu y\cdot \xi)^{n}}{n!}
dy -1 +a_{\min}\\
&\geq m_0+ m_{2}(\xi)\mu^{2}+\sum_{n=2}^{\infty}m_{2n}(\xi)\mu^{2n}-1+a_{\min}
\end{align*}
Let $m:=\ds\inf_{\xi\in
S^{N-1}}m_{2}(\xi)(>0)$.
We then have
 $$\frac{\lambda_0(\xi,\mu,a)}{\mu}\ge \frac{m_0+m\mu^2-1+a_{\min}}{\mu} \to \infty$$
  as $\mu\to\infty$. By
$\lambda_0(\xi,0,a)>0$,
$$\frac{\lambda_0(\xi,\mu,a)}{\mu}\to\infty$$
 as $\mu\to 0+$. This together with the smoothness of
$\lambda_0(\xi,\mu,a)$ (see Theorem \ref{PE-simplicity-thm}) implies that there is $\mu^*(\xi)$ such that
\eqref{mu-star-eq} holds.
\end{proof}

\begin{proposition}
\label{PE-convex-prop} For given $\xi\in S^{N-1}$,
suppose that $\lambda_0(\xi,\mu,a)$ is the principal eigenvalue of $-\p_t+\mathcal{K}_{\xi,\mu}-I+a(\cdot,\cdot)I$  for all $\mu\in\RR$. Then
$\lambda_0(\xi,\mu,a)$ is convex in $\mu$.
\end{proposition}

\begin{proof}
First, recall that $\Phi(t;\xi,\mu,a)$ is the solution operator of \eqref{linear-new-eq1}. Let
$$
\Phi^p(T;\xi,\mu,a)=\Phi(T;\xi,\mu,a)|_{X_p}.
$$
 By \cite[Proposition 3.10]{RaSh},  we have
 $$r(\Phi^p(T;\xi,\mu,a))=e^{\lambda_0(\xi,\mu,a)T}.
 $$

Note that $\Phi(t;\xi,0,a)$ is independent of $\xi\in S^{N-1}$.
 We put
\begin{equation}
\label{tilde-phi-eq} \tilde \Phi(t;a)= \Phi(t;\xi,0,a)
\end{equation}
for $\xi\in S^{N-1}$.
For given $u_0\in X$ and $\mu\in\RR$,
  letting $u_0^{\xi,\mu}(x)=e^{-\mu x\cdot
\xi}u_0(x)$, then $u_0^{\xi,\mu}\in X(|\mu|)$. By the uniqueness of
solutions of \eqref{linear-new-eq1}, we have that
for given $u_0\in X$, $\xi\in S^{N-1}$, and $\mu\in\RR$,
\begin{equation}
\label{convex-eq1}
\Phi(t;\xi,\mu,a)u_0=e^{\mu x\cdot
\xi}\tilde \Phi(t;a)u_0^{\xi,\mu}.
\end{equation}

 Next, observe that for each
$x\in\RR^N$, there is a measure $m(x;y,dy)$ such that
\begin{equation}
\label{measure-eq1}
 (\tilde\Phi(T;a)u_0)(x)=\int_{\RR^N}u_0(y)m(x;y,dy).
\end{equation}
Moreover, by $(\tilde\Phi(T;a)u_0(\cdot-p_ie_i))(x)=(\tilde \Phi(T;a) u_0(\cdot))(x-p_ie_i)$ for $x\in\RR^N$ and
$i=1,2,\cdots,N$,
\begin{align*}
 \int_{\RR^N} u_0(y)m(x-p_ie_i;y,dy)&=\int_{\RR^N}
u_0(y-p_ie_i)m(x;y,dy)\\
&=\int_{\RR^N} u_0(y)m(x;y+p_ie_i,dy)
\end{align*}
and hence
\begin{equation}
\label{measure-eq2} m(x-p_ie_i;y,dy)=m(x;y+p_ie_i,dy)
\end{equation}
for $i=1,2,\cdots,N$.
By \eqref{convex-eq1}, we have
$$
(\Phi(T;\xi,\mu,a)u_0)(x)=\int_{\RR^N} e^{\mu (x-y)\cdot\xi}u_0(y)m(x;y,dy),\quad u_0\in X.
$$

Let
  $\hat{\lambda}_{0}(\mu_{i}):=r(\Phi^p(T;\xi,\mu_i))$.
  By the arguments of \cite[Theorem A (2)]{ShZh0},
$$\ln[\hat{\lambda}_{0}(\mu_{1})]^{\alpha}[\hat{\lambda}_{0}(\mu_{2})]^{1-\alpha} \geq
\ln(r(\Phi^p(T;\xi,\alpha\mu_{1}+(1-\alpha)\mu_{2})).
$$
Thus, by $r(\Phi^p(T;\xi,\mu,a))=e^{\lambda_0(\xi,\mu,a)T}$, we have
$$\alpha \lambda_{0}(\xi,\mu_{1},a)+(1-\alpha)
\lambda_{0}(\xi,\mu_{2},a) \geq \lambda_{0}(\xi,\alpha\mu_{1}+(1-\alpha)\mu_{2},a),
$$
that is, $\lambda_{0}(\xi,\mu,a)$ is convex in $\mu$.

\end{proof}

For a fixed $\xi\in S^{N-1}$ and $a\in \mathcal{X}_p$, we may denote $\lambda_0(\xi,\mu,a)$ by
$\lambda(\mu)$.

\begin{proposition}
\label{PE-useful-prop} Let $\xi\in S^{N-1}$ and $a\in\mathcal{X}_p$  be given. Assume that \eqref{eigenvalue-eq} has the principal
eigenvalue $\lambda(\mu)$ for $\mu\in\RR$ and that $\lambda(0)>0$. Then we have:

\begin{itemize}
\item[(i)]  $\lambda^{'}(\mu)<\frac{\lambda(\mu)}{\mu}$ for $0<\mu<\mu^*(\xi).$

\item[(ii)] For every $\epsilon>0$, there exists some $\mu_\epsilon>0$ such that for $\mu_\epsilon < \mu <
\mu^*(\xi)$,$$ -\lambda^{'}(\mu)< -\frac{\lambda(\mu^*(\xi))}{\mu^*(\xi)}+\epsilon.
$$
\end{itemize}
\end{proposition}

\begin{proof}
It follows from Theorem \ref{PE-simplicity-thm}, Propositions \ref{PE-minimum-prop}, \ref{PE-convex-prop}, and
the arguments of \cite[Theorem 3.1]{ShZh0}.
\end{proof}

\begin{proposition}
\label{PE-perturbation-prop}
 For any $\epsilon>0$ and $M>0$, there are $a^\pm(\cdot,\cdot)$ satisfying the vanishing condition in Proposition \ref{PE-sufficient-prop}
  such that
$$
a(t,x)-\epsilon\le a^-(t,x)\leq a(t,x)\leq a^+(t,x)\le a(t,x)+\epsilon
$$
and
$$
|\lambda_0(\xi,\mu,a)-\lambda_0(\xi,\mu,a^\pm)|<\epsilon
$$
for $\xi\in S^{N-1}$ and $|\mu|\leq M$.
\end{proposition}

\begin{proof}
It follows from \cite[Lemma 4.1]{RaSh} and the fact that
$$
\Phi^p(T;\xi,\mu,a\pm\epsilon)=e^{\pm\epsilon T}\Phi^p(T;\xi,\mu,a).
$$
\end{proof}

\section{Spreading Speeds}

In this section, we investigate the existence and characterization of the spreading speeds of
\eqref{main-eq}.

Throughout this section, we assume (H1) and (H2).  $u(t,x;u_0)$ denotes the solution of \eqref{main-eq} with $u(0,x;u_0)=u_0(x)$.
As  mentioned in the introduction, by \cite[Theorem E]{RaSh},
\eqref{main-eq} has a unique positive periodic solution $u^*(\cdot,\cdot)\in \mathcal{X}_p^+$.

We first recall the notion of spreading speed intervals and spreading speeds introduced in \cite{ShZh0}.

\begin{definition}
\label{spreading-speed-interval-def}
 For a given
 vector $\xi\in S^{N-1}$, let
\begin{equation*}
C_{\rm inf}^*(\xi)= \Big\{\;c\;|\;\forall\; u_0\in X^+(\xi),\,\limsup_{t\to\infty}\sup_{x\cdot\xi \leq ct}|u(t,x;u_0)-u^*(t,x)|= 0\Big\}
\end{equation*}
and
\begin{equation*}
C_{\rm sup}^*(\xi)=\Big\{\;c\;|\; \forall\; u_0\in X^+(\xi),\,
  \limsup_{t\to\infty}\sup_{x\cdot\xi \geq ct}u(t,x;u_0)=0\Big\}.
\end{equation*}

Define
\begin{eqnarray*}
c_{\rm inf}^*(\xi)=\sup\;\{\;c\;|\;c\in C_{\rm inf}^*(\xi)\},\quad c_{\rm sup}^*(\xi)=\inf\;\{\;c\;|\;c\in
C_{\rm sup}^*(\xi)\}.
\end{eqnarray*}
We call $[c_{\rm inf}^*(\xi),c_{\rm sup}^*(\xi)]$  the {\rm spreading speed interval} of \eqref{main-eq} in the
direction of $\xi$.
If $c_{\rm inf}^*(\xi)=c_{\rm sup}^*(\xi)$, we call $c^*(\xi):=c_{\rm inf}^*(\xi)$ the {\rm spreading speed} of \eqref{main-eq}
in the direction of $\xi$.
\end{definition}

Observe  that $X^+(\xi)$ is not empty (see \eqref{X-space-in-direction-xi}). If $c_1\in C_{\inf}^*(\xi)$ and $c_2\in C_{\sup}^*(\xi)$, then $c_1<c_2$, and
for any $c^{'}<c_1$ and $c^{''}>c_2$, $c^{'}\in C_{\inf}^*(\xi)$ and $c^{''}\in C_{\sup}^*(\xi)$. By the arguments in \cite[Corollary 4.1]{ShZh0},
both $C_{\inf}^*(\xi)$ and $C_{\sup}^*(\xi)$ are not empty. Hence $[c_{\inf}^*(\xi),c_{\sup}^*(\xi)]$ is well defined.

The main results of this section are stated in the following theorems.

 \begin{theorem} [Existence of spreading speeds]
 \label{spreading-thm1}
 Assume (H1) and (H2). For any given $\xi\in S^{N-1}$,
 $c_{\inf}^*(\xi)=c_{\sup}^*(\xi)$ and hence the spreading speed  $c^*(\xi)$ of \eqref{main-eq} in the direction of $\xi$ exists.
 Moreover,
 $$
 c^*(\xi)=\inf_{\mu>0}\frac{\lambda_0(\xi,\mu,a_0)}{\mu},
 $$
 where $a_0(t,x)=f(t,x,0)$.
 \end{theorem}

\begin{theorem}[Spreading features of  spreading speeds]
\label{spreading-thm2}
  Assume (H1) and (H2).
\begin{itemize}
\item[(1)] If $ u_0\in X^+$ satisfies that  $u_0(x)=0$ for $x\in\RR^N$ with $|x\cdot\xi|\gg 1$, then
$$
\limsup_{t\to\infty}\sup_{x\cdot\xi\ge ct}u(t,x;u_0)=0\quad \forall\,\, c>c^*(\xi)
$$
and
$$\limsup_{t\to\infty}\sup_{x\cdot\xi\le - ct}u(t,x;u_0)=0\quad \forall\,\, c>c^*(-\xi).$$

\item[(2)]  For any $\sigma>0$,  $r>0$, and $u_0\in X^+$ satisfying $u_0(x)\geq\sigma$ for all $x\in\RR^N$ with $|x\cdot \xi|\leq r$
$$\liminf_{t\to\infty}\inf_{0\le x\cdot \xi\leq ct}(u(t,x;u_0)-u^*(t,x))=0 \quad \forall\,\, 0<c<c^*(\xi)$$
and
$$\liminf_{t\to\infty}\inf_{-ct\le x\cdot \xi\leq 0}(u(t,x;u_0)-u^*(t,x))=0 \quad \forall \,\, 0<c<c^*(-\xi).$$

\item[(3)] If $u_0\in X^+$ satisfies that  $u_0(x)=0$ for $x\in\RR^N$ with $\|x\|\gg 1$, then
$$\limsup_{t\to\infty}\sup_{\|x\|\geq ct}u(t,x;u_0)=0$$
for all $c>\sup_{\xi\in S^{N-1}}c^*(\xi)$.

\item[(4)]  Assume that
$0<c<\inf_{\xi\in S^{N-1}}\{c^*(\xi)\}$. Then for any $\sigma>0$ and $r>0$,
$$\liminf_{t\to\infty}\inf_{\|x\|\leq ct}(u(t,x;u_0)-u^*(t,x))=0 $$ for every $u_0 \in X^+$ satisfying
 $u_0(x)\geq\sigma$ for $x\in\RR^N$ with  $\|x\|\leq r$.
\end{itemize}
\end{theorem}

To prove the above theorems, we first prove some lemmas.

Consider the space shifted equations of \eqref{main-eq},
\begin{equation}
\label{main-shifted-eq}
 \frac{\p u}{\p
 t}=\int_{\RR^N}k(y-x)u(t,y)dy-u(t,x)+u(t,x)f(t, x+z,u(t,x)),\quad
 x\in\RR^N,
\end{equation}
where $z\in\RR^N$. Let $u(t,x;u_0,z)$ be the solution of \eqref{main-shifted-eq} with
$u(0,x;u_0,z)=u_0(x)$ for $u_0\in X$.

\begin{lemma}
\label{main-lm1}
\begin{itemize}
\item[(1)]
Let $\xi\in S^{N-1}$,
 $u_0(\cdot,z)\in \tilde X^+$ with $\ds\liminf _{x\cdot \xi\to -\infty}u_0(x,z)>0$ and
$\ds\limsup_{x\cdot\xi\to\infty}u_0(x,z)=0$ for all $z\in\RR^N$, and $c\in\RR$ be given.
 If there is
$\delta_0$  such that
\begin{equation}
\label{global-eq1} \liminf_{x\cdot\xi\leq cnT,n\to\infty}u(nT,x;u_0(\cdot,z),z)\geq \delta_0\quad \text{\rm uniformly
in}\quad  z\in\RR^N,
\end{equation}
 then for every $c^{'}<c$,
$$
\liminf_{x\cdot\xi\leq c^{'}t,t\to\infty}(u(t,x;u_0(\cdot,z),z)-u^*(t,x+z))=0 \quad \text{\rm uniformly in}\quad
z\in\RR^N.
$$

\item[(2)] Let $c\in\RR$ and
$u_0(\cdot,z)\in \tilde X$ with $u_0(\cdot,z)\geq 0$ ($z\in\RR^N$) be given. If there is $\delta_0$  such that
\begin{equation}
\label{global-eq2} \liminf_{|x\cdot\xi|\leq cnT,n\to\infty}u(nT,x;u_0(\cdot,z),z)\geq \delta_0 \quad \text{\rm
uniformly in}\quad  z\in\RR^N,
\end{equation}
 then for every $c^{'}<c$,
$$
\limsup_{|x\cdot\xi|\leq c^{'}t,t\to\infty} |u(t,x;u_0(\cdot,z),z)-u^*(t,x+z)|=0 \quad \text{\rm uniformly in}\quad
z\in\RR^N.
$$

\item[(3)]  Let $c\in\RR$ and
$u_0(\cdot,z)\in \tilde X$ with $u_0(\cdot,z)\geq 0$ ($z\in\RR^N$) be given. If there is $\delta_0$  such that
\begin{equation}
\label{global-eq3} \liminf_{\|x\|\leq cnT,n\to\infty}u(nT,x;u_0(\cdot,z),z)\geq \delta_0 \quad \text{\rm uniformly
in}\quad  z\in\RR^N,
\end{equation}
then for every $c^{'}<c$,
$$
\limsup_{\|x\|\leq c^{'}t,t\to\infty} |u(t,x;u_0(\cdot,z),z)-u^*(t,x+z)|=0 \quad \text{\rm uniformly in}\quad  z\in\RR^N.
$$
\end{itemize}
\end{lemma}

\begin{proof}
It follows from the arguments of \cite[Proposition 4.4]{ShZh0}.
\end{proof}

\begin{lemma}
\label{main-lm2}
$$
\int_{\|y-x\|\geq B}e^{\mu\|y-x\|}m(x;y,dy)\to 0\quad {\rm as}\quad B\to\infty
$$
uniformly for $\mu$ in bounded sets and for $x\in\RR^N$.
\end{lemma}

\begin{proof}
For given $\mu_0>0$ and $n\in\NN$, let $u_n\in X(\mu_0+1)$ be such that
$$
u_n(x)=\begin{cases} e^{\mu_0\|x\|}\quad {\rm for}\quad \|x\|\geq n\cr 0\quad {\rm for}\quad \|x\|\leq n-1
\end{cases}
$$
and
$$
0\leq u_n(x)\leq e^{\mu_0 n} \quad {\rm for}\quad \|x\|\leq n.
$$
Then $\|u_n\|_{X(\mu_0+1)}\to 0$ as $n\to\infty$. Therefore, $\|\tilde\Phi(T)u_n\|_{X(\mu_0+1)}\to 0$  as
$n\to\infty.$ This  implies that
$$
\int_{\RR^N}u_n(y)m(x;y,dy)\to 0\quad {\rm as}\quad n\to\infty
$$
uniformly for $x$ in bounded subsets of $\RR^N$ and then
$$
\int_{\|y\|\geq n}e^{\mu_0 \|y\|}m(x;y,dy)\to 0\quad {\rm as}\quad n\to\infty
$$
uniformly for $x$ in bounded subsets of $\RR^N$. The later implies that
$$
\int_{\|y-x\|\geq n}e^{\mu\|y-x\|}m(x;y,dy)\to 0\quad {\rm as}\quad n\to\infty
$$
uniformly for $|\mu|\leq\mu_0$ and $x$ in bounded subset of $\RR^N$. By \eqref{measure-eq2}, for every $1\leq
i\leq N$,
\begin{align*}
&\int_{\|y-(x+p_ie_i)\|\geq n}e^{\mu\|y-(x+p_ie_i)\|}m(x+p_ie_i;y,dy)\\
&= \int_{\|y-x\|\geq
n}e^{\mu\|y-x\|}m(x+p_ie_i;y+p_ie_i,dy)\\
&=\int_{\|y-x\|\geq n}e^{\mu\|y-x\|}m(x;y,dy).
\end{align*}
We then have
$$
\int_{\|y-x\|\geq n}e^{\mu\|y-x\|}m(x;y,dy)\to 0\quad {\rm as}\quad n\to\infty
$$
uniformly for $|\mu|\leq\mu_0$ and $x\in\RR^N$. The lemma now follows.
 \end{proof}

Without loss of generality, in the rest of this section, we assume that the time period $T=1$.

\begin{lemma}
\label{main-lm3}
For given $\xi\in S^{N-1}$, if $\lambda_0(\xi,\mu,a_0)$ is the principal eigenvalue of
$-\p_t+\mathcal{K}_{\xi,\mu}-I+a_0(\cdot,\cdot)I$ for any $\mu>0$, then
\begin{equation}
\label{ineq1}
c_{\sup}^*(\xi)\le\inf_{\mu>0}\frac{\lambda_0(\xi,\mu,a_0)}{\mu}.
\end{equation}
\end{lemma}

\begin{proof}
For given $\xi\in S^{N-1}$, put $\lambda(\mu)=\lambda_0(\xi,\mu,a_0)$.
For any $\mu>0$, suppose that $\phi(\mu,\cdot,\cdot)\in\mathcal{X}_p^+$, $\|\phi(\mu,\cdot,\cdot)\|=1$,  and
 $$[-\p_t+(\mathcal{K}_{\xi,\mu}-I+a_0(\cdot,\cdot)I)]{\phi}(\mu,t,x)=\lambda(\mu){\phi}(\mu,t,x).$$
 Since $f(t,x,u)=f(t,x,0)+f_{u}(t,x,\eta)u$ for some $0 \leq \eta \leq
u$, we have, by assumption (H1), $f(t,x,u) \leq f(t,x,0)$ for $u \geq 0$. If $u_0\in X^+$ , then
\begin{equation}
\label{aux-eq12} u(t,x;u_0,z)\leq (\Phi(t;\xi,0,a_0(\cdot,\cdot+z))u_0)(x)\quad {\rm for}\quad x,z\in\RR^N,
\end{equation}
 where $u(t,x;u_0,z)$ is the solution of the space shifted equation \eqref{main-shifted-eq} of \eqref{main-eq}.
It can easily be verified that
\begin{align*}
(\Phi(t;\xi,0,a_0(\cdot,\cdot+z))\tilde{u}_{0})(x)&=Me^{-\mu(x\cdot \xi-\tilde{c}t)}{\phi}(\mu,t,x+z)
\end{align*}
 with $\tilde{u}_{0}(x)=Me^{-\mu x\cdot \xi}{\phi}(\mu, 0,x+z)$
for $\tilde{c}=\frac{\lambda(\mu)}{\mu}$ and $M > 0$.
By the definition of $X^+(\xi)$ (see \eqref{X-space-in-direction-xi}), for any $u_0\in X^+(\xi)$, we can
choose $M > 0$ large enough such that $\tilde{u}_{0} \geq u_{0}$. Then by Propositions \ref{comparison-linear-prop} and
\ref{comparison-nonlinear-prop},  we have
\begin{align*}
u(t,x;u_{0},z) &\leq (\Phi(t;\xi,0,a_0(\cdot,\cdot+z))u_{0})(x)\\
&\leq (\Phi(t;\xi,0,a_0(\cdot,\cdot+z))\tilde{u}_{0})(x)\\
&=Me^{-\mu(x\cdot
\xi-\tilde{c}t)}{\phi}(\mu,t,x+z).
\end{align*}
 Hence,
 $$\limsup_{x\cdot\xi\geq ct,t\to\infty}u(t,x;u_0,z)=0\quad{\rm for\quad
every}\quad c>\tilde c$$
uniformly in $z\in\RR$.
 This   implies that $c_{\sup}^*(\xi)\leq \frac{\lambda(\mu)}{\mu}$ for any $\mu>0$ and hence
 \eqref{ineq1} holds.
\end{proof}

\begin{lemma}
\label{main-lm4}
For given $\xi\in S^{N-1}$, if $\lambda_0(\xi,\mu,a_0)$ is the principal eigenvalue of
$-\p_t+\mathcal{K}_{\xi,\mu}-I+a_0(\cdot,\cdot)I$ for any $\mu>0$, then
\begin{equation}
\label{ineq2}
c_{\inf}^*(\xi)\ge\inf_{\mu>0}\frac{\lambda_0(\xi,\mu,a_0)}{\mu}.
\end{equation}
\end{lemma}

\begin{proof}
We  prove \eqref{ineq2}  by modifying the
arguments in  \cite{Wei1}.

Observe that, for every $\epsilon_0>0$, there is  $b_0>0$ such that
\begin{equation}
\label{aux-eq100} f(t,x,u)\geq f(t,x,0)-\epsilon_0\quad {\rm for}\quad 0\leq u\leq b_0,\quad x\in\RR^N.
\end{equation}
 Hence if $u_0\in X^+$ is so small
that $0\leq u(t,x;u_0,z)\leq b_0$ for $t\in [0,1]$, $x\in\RR^N$ and $z\in\RR^N$, then
\begin{equation}
\label{aux-eq200} u(1,x;u_0,z)\geq e^{-\epsilon_0}(\Phi(1;\xi,0,a_0(\cdot,\cdot+z))u_0)(x)
\end{equation}
for $x\in\RR^N$ and $z\in\RR^N$.

Let $r(\mu)$ be the spectral radius of $\Phi^p(1;\xi,\mu,a_0)$. Then  $\lambda(\mu)=\ln r(\mu)$ and
$r(\mu)$ is an eigenvalue of $\Phi^p(1;\xi,\mu,a_0(\cdot,\cdot))$
with a positive eigenfunction
$\phi(\mu,x):=\phi(\mu,1,x)$, where $\phi(\mu,t,x)$ is as in the proof of Lemma \ref{main-lm3}.

 By Proposition \ref{PE-useful-prop},   for any $\epsilon_1>0$, there is $\mu_{\epsilon_1}$ such
that
\begin{equation}\label{estimate-eq3}
-\lambda^{'}(\mu)<-\frac{\lambda(\mu^*(\xi))}{\mu^*(\xi)}+\epsilon_1
\end{equation}
for $\mu_{\epsilon_{1}}<\mu<\mu^*(\xi)$. In the following, we fix $\mu\in (\mu_{\epsilon_1},\mu^*(\xi))$. By
Proposition \ref{PE-useful-prop}  again, we can choose $\epsilon_0>0$ so small that
\begin{equation}
\label{estimate-eq1} \lambda(\mu)-\mu \lambda^{'}(\mu)-3\epsilon_{0}>0.
\end{equation}

Let $\zeta:\RR\to [0,1]$ be a smooth function satisfying that
 \begin{equation}
 \label{truncated-function}
 \zeta(s)=\begin{cases} 1\quad {\rm for}\quad  |s|\leq 1\cr
  0\quad {\rm for}\quad  |s|\geq 2.
  \end{cases}
  \end{equation}
  By Theorem \ref{PE-simplicity-thm}, $\phi(\mu,x)$ is smooth in $\mu$.
 Let
$$
\kappa(\mu,z)=\frac{\phi_\mu(\mu,z)}{\phi(\mu,z)}.
$$
For given $\gamma>0$, $B>0$,  and $z\in\RR^N$, define
\begin{align*}
&\tau(\mu,\gamma,z,B)\\
&=\frac{1}{\gamma}tan^{-1}\frac{\int_{\RR^N}\phi(\mu,y)e^{-\mu
(y-z)\cdot\xi}\sin\gamma(-(y-z)\cdot\xi+\kappa(\mu,y))\zeta(\|y-z\|/B)m(z;y,dy)}
{\int_{\RR^N}\phi(\mu,y)e^{-\mu(y-z)\cdot\xi}\cos\gamma(-(y-z)\cdot\xi+\kappa(\mu,y))\zeta(\|y-z\|/B)m(z;y,dy)}.
\end{align*}
By Lemma \ref{main-lm2}, $\tau(\mu,\gamma,z,B)$ is well defined for any $B>0$ and $0<\gamma \ll 1$, and
\begin{align*}
&\lim_{\gamma\to 0}\tau(\mu,\gamma,z,B)\\
&=\frac{\int_{\RR^N} \phi(\mu,y)e^{-\mu
(y-z)\cdot\xi} \big(-(y-z)\cdot\xi+\kappa(\mu,y)\big)\zeta(\|y-z\|/B)m(z;y,dy)}
{\int_{\RR^N} \phi(\mu,y)e^{-\mu
(y-z)\cdot\xi}\zeta(\|y-z\|/B)m(z;y,dy)}
\end{align*}
uniformly in $z\in\RR^N$ and $B>0$.
By Lemma \ref{main-lm2} again,
\begin{equation}
\label{new-eq1}
\lim_{B\to\infty} \int_{\RR^N} \phi(\mu,y)e^{-\mu
(y-z)\cdot\xi}\zeta(\|y-z\|/B)m(z;y,dy)=r(\mu)\phi(\mu,z)
\end{equation}
uniformly in $z\in\RR^N$
and
\begin{align}
\label{new-eq2}
&\lim_{B\to\infty}\Big[ \int_{\RR^N} \phi(\mu,y)e^{-\mu
(y-z)\cdot\xi}\big(-(y-z)\cdot\xi\big)\zeta(\|y-z\|/B)m(z;y,dy)\nonumber\\
&\quad \quad +\int_{\RR^N} \phi_{\mu}(\mu,y)e^{-\mu
(y-z)\cdot\xi}\zeta(\|y-z\|/B)m(z;y,dy)\Big]\nonumber\\
&=  \int_{\RR^N} \phi(\mu,y)e^{-\mu
(y-z)\cdot\xi}\big(-(y-z)\cdot\xi\big)m(z;y,dy)\nonumber\\
&\qquad +\int_{\RR^N} \phi_{\mu}(\mu,y)e^{-\mu
(y-z)\cdot\xi}m(z;y,dy)\nonumber\\
&=r^{'}(\mu)\phi(\mu,z)+r(\mu)\phi_\mu(\mu,z)
\end{align}
uniformly in $z\in\RR^N$.

By \eqref{new-eq1} and \eqref{new-eq2}, we can choose
 $B\gg 1$ and fix it   so that for any $0<\gamma\ll 1$ and $z,z'\in\RR^N$,
\begin{align}
\label{B-eq}
&\int_{\RR^N}\phi(\mu,y)e^{-\mu(y-z)\cdot\xi} \cos\gamma\big(-(y-z)\cdot\xi +\kappa(\mu,y)\big)\cdot \zeta(\|y-z\|/B)m(z;y,dy)\nonumber\\
&\qquad \ge e^{\lambda(\mu)-\epsilon_0}\phi(\mu,z),
\end{align}
\begin{equation}
\label{augg-eq1}
\gamma(2B+|\tau(\mu,\gamma,z,B)|+|\kappa(\mu,z')|)<\pi,
\end{equation}
\begin{equation}
\label{new-eq3} -\kappa(\mu,z)+\tau(\mu,\gamma,z,B)<\lambda^{'}(\mu)+\frac{\epsilon_0}{\mu},
\end{equation}
and
\begin{equation}
\label{new-eq4} \kappa(\mu,z)-\tau(\mu,\gamma,z,B)<-\lambda^{'}(\mu)+\epsilon_1.
\end{equation}

For given $\epsilon_2>0$ and $\gamma>0$, define
\begin{equation}
\label{aux-eq3} v(s,z)=\begin{cases} \epsilon_2\phi(\mu,z)e^{-\mu s}\sin \gamma(s-\kappa(\mu,z)),\quad 0\leq
s-\kappa(\mu,z)\leq\frac{\pi}{\gamma}\cr\cr
 0,\quad {\rm otherwise}.
\end{cases}
\end{equation}
Let
\begin{equation}
\label{augg-eq0}
v^*(x;s,z)=v(x\cdot\xi+s-\kappa(\mu,z)+\tau(\mu,\gamma,z,B),x+z).
\end{equation}
Choose $\epsilon_2>0$ so small that
$$
0\leq u(t,x; v^*(\cdot;s,z),z)\leq b_0\quad {\rm for}\quad t\in [0,1],\quad x,z\in\RR^N.
$$
Let
$$
\eta(\gamma,\mu,z,B)=-\kappa(\mu,z)+\tau(\mu,\gamma,z,B).
$$
Observe that
$$
\big(\Phi(1;\xi,0,a_0(\cdot,\cdot+z))v^*(\cdot;s,z))(x)=\int_{\RR^N} v^*(y-z;s,z)m(x+z;y,dy).
$$
Observe also that
\begin{align}
\label{augg-eq2}
&v^*(y-z;s,z)\nonumber\\
&=v\big((y-z)\cdot \xi+s+\eta(\gamma,\mu,z,B),y\big)\nonumber\\
&=\begin{cases}
\epsilon_2 \phi(\mu,y)e^{-\mu \big((y-z)\cdot \xi+s+\eta(\gamma,\mu,z,B)\big)}\sin \gamma((y-z)\cdot \xi+s+\eta(\gamma,\mu,z,B)-\kappa(\mu,y))\cr
\qquad\qquad\qquad\quad  {\rm if}\,\, 0\le (y-z)\cdot \xi+s+\eta(\gamma,\mu,z,B)-\kappa(\mu,y)\le\frac{\pi}{\gamma}\cr\cr
0\quad {\rm otherwise},
\end{cases}
\end{align}
\begin{align}
\label{augg-eq3}
-\pi&\le \gamma\Big[(y-z)\cdot\xi+s+\eta(\gamma,\mu,z,B)-\kappa(\mu,y)\Big]\nonumber\\
&\le 2\pi\quad {\rm for}\,\, 0\le s-\kappa(\mu,z)\le \frac{\pi}{\gamma},\,\, \|y-z\|\le 2B,
\end{align}
and
\begin{align}
\label{augg-eq4}
&\int_{\RR^N} \phi(\mu,y)e^{-\mu(y-z)\cdot\xi}\sin \gamma((y-z)\cdot \xi+s+\eta(\gamma,\mu,z,B)-\kappa(\mu,y))\zeta(\|y-z\|/B)m(z,y,dy)\nonumber\\
&=\int_{\RR^N} \phi(\mu,y) e^{-\mu(y-z)\cdot\xi} \Big[\cos\gamma(\kappa(\mu,y)-(y-z)\cdot\xi)\sin\gamma(s+\eta(\gamma,\mu,z,B))\nonumber\\
&\qquad\qquad  -\sin\gamma(\kappa(\mu,y)-(y-z)\cdot\xi)\cos\gamma(s+\eta(\gamma,\mu,z,B))\Big]\zeta(\|y-z\|/B)m(z,y,dy)\nonumber\\
&=\int_{\RR^N}\phi(\mu,y)e^{-\mu(y-z)\cdot\xi}\cos\gamma(\kappa(\mu,y)-(y-z)\cdot\xi) \zeta(\|y-z\|/B) m(z,y,dy)\nonumber\\
&\qquad \cdot \Big[\sin \gamma(s+\eta(\gamma,\mu,z,B))\nonumber -\cos\gamma(s+\eta(\gamma,\mu,z,B))\nonumber\\
&\qquad \cdot \frac{\int_{\RR^N} \phi(\mu,y)e^{-\mu(y-z)\cdot\xi}\sin \gamma(\kappa(\mu,y)-(y-z)\cdot\xi))\zeta(\|y-z\|/B)m(z,y,dy)}{\int_{\RR^N}
 \phi(\mu,y)e^{-\mu(y-z)\cdot\xi}\cos \gamma(\kappa(\mu,y)-(y-z)\cdot\xi))\zeta(\|y-z\|/B)m(z,y,dy)}\Big]\nonumber\\
 &=\int_{\RR^N}\phi(\mu,y)e^{-\mu(y-z)\cdot\xi}\cos \gamma(\kappa(\mu,y)-(y-z)\cdot\xi) \zeta(\|y-z\|/B)m(z,y,dy))\nonumber\\
 &\qquad \cdot \Big[\sin \gamma(s+\eta(\gamma,\mu,z,B)) -\cos\gamma(s+\eta(\gamma,\mu,z,B))\tan\gamma\tau(\mu,\gamma,z,B)\Big]\nonumber\\
 &=\frac{\sin \gamma(s-\kappa(\mu,z))}{\cos\gamma\tau(\mu,\gamma,z,B)}\nonumber\\
 &\qquad \cdot \int_{\RR^N}\phi(\mu,y)e^{-\mu(y-z)\cdot\xi}\cos\gamma(\kappa(\mu,y)-(y-z)\cdot\xi) \zeta(\|y-z\|/B)m(z,y,dy)
\end{align}
Then for $0\leq s-\kappa(\mu,z)\leq\frac{\pi}{\gamma}$, we have
\begin{align}
\label{augg-eq5}
&u(1,0;v^*(\cdot;s,z),z)\nonumber\\
&\geq e^{-\epsilon_{0}}\big(\Phi(1;\xi,0,a_0(\cdot,\cdot+z))v^{*}(\cdot;s,z)\big)(0)\qquad\qquad {\rm (by\,\,\, \eqref{aux-eq200})}\nonumber\\
&\geq\epsilon_2 e^{-\epsilon_{0}}\int_{\RR^N} \Big[\phi(\mu,y)e^{-\mu[(y-z)\cdot\xi+s+\eta(\gamma,\mu,z.B)]}\nonumber\\
&\qquad \cdot
\sin\gamma[(y-z)\cdot\xi+s+\eta(\gamma,\mu,z,B)-\kappa(\mu,y)]\nonumber\\
&\quad\quad \cdot
\zeta(\|y-z\|/B)\Big]m(z;y,dy)\qquad \qquad\qquad\qquad  {\rm (by \,\, \eqref{augg-eq2}, \,\, \eqref{augg-eq3})}\nonumber\\
&=e^{-\epsilon_{0}}v(s,z) e^{-\mu\eta(\gamma,\mu,z,B)}\frac{\sec\gamma\tau(\mu,\gamma,z,B)}{\phi(\mu,z)}\nonumber\\
&\qquad \cdot \int_{\RR^N}
\Big[\phi(\mu,y)e^{-\mu(y-z)\cdot\xi}\cdot \cos\gamma(-(y-z)\cdot\xi+\kappa(\mu,y))\nonumber\\
&\quad\quad\cdot \zeta(\|y-z\|/B)\Big]m(z;y,dy) \qquad\qquad\qquad\qquad {\rm (by \,\, \eqref{augg-eq4}).}
\end{align}
Observe that
\begin{align}
\label{augg-eq6}
&\lim_{\gamma\to 0}e^{-\epsilon_{0}}e^{-\mu\eta(\gamma,\mu,z,B)}\frac{\sec\gamma\tau(\mu,\gamma,z,B)}{\phi(\mu,z)}
\int_{\RR^N}\Big[\phi(\mu,y)e^{-\mu(y-z)\cdot\xi}\nonumber\\
&\qquad \cdot\cos\gamma(-(y-z)\cdot\xi+\kappa(\mu,y))\cdot \zeta(\|y-z\|/B)\Big]m(z;y,dy)\nonumber\\
&\ge e^{-\epsilon_{0}}e^{-\mu \lambda^{'}(\mu)-\epsilon_0}  e^{\lambda(\mu)-\epsilon_0}\qquad \qquad\text{(by}\,\, \eqref{B-eq},\,\,\eqref{new-eq3})\nonumber\\
&=e^{\lambda(\mu)-\mu \lambda^{'}(\mu)-3\epsilon_{0}}\nonumber\\
&>1 \qquad\qquad\qquad\qquad\qquad\qquad \quad\text{\rm (by}\,\, \eqref{estimate-eq1}\text{\rm )}.
\end{align}
By \eqref{augg-eq0}, \eqref{augg-eq5}, and \eqref{augg-eq6},  for $0\leq s-\kappa(\mu,z)\leq \frac{\pi}{\gamma}$ and $0<\gamma\ll 1$,
\begin{align}
\label{augg-eq7}
u(1,0;v^*(\cdot;s,z),z)&\geq v(s,z)\nonumber\\
&= v^*((\kappa(\mu,z)-\tau(\mu,\gamma,z,B))\xi;s,(-\kappa(\mu,z)+\tau(\mu,\gamma,z,B))\xi+z)
\end{align}
Since $v(s,z)=0$ for $s\le \kappa(\mu,z)$ or $s\ge \kappa(\mu,z)+\frac{\pi}{\gamma}$, we have
that \eqref{augg-eq7} holds
for all $s\in\RR$ and $z\in\RR^N$.

Let $\bar s(z)$ be such that $v(\bar s(z),z)=\max_{s\in\RR} v(s,z)$. Let
$$
\bar v(s,z)=\begin{cases} v(\bar s(z),z),\quad s\leq\bar s(z)-\frac{\pi}{\gamma}\cr
v(s+\frac{\pi}{\gamma},z),\quad s\geq \bar s(z)-\frac{\pi}{\gamma}.
\end{cases}
$$
Set
$$
\bar v^*(x;s,z)=\bar v(x\cdot\xi+s-\kappa(\mu,z)+\tau(\mu,\gamma,z,B),x+z).
$$
We then have
\begin{equation}
\label{augg-eq8}
\begin{cases}
\bar v(s,z) = v(s+\frac{\pi}{\gamma},z)\quad \forall\,\, s\ge s(z)-\frac{\pi}{\gamma},\,\, z\in\RR^N\cr
\bar v(s,z)=v(\bar s(z),z)\ge v(s+\frac{\pi}{\gamma},z)\quad \forall \,\, s\le \bar s(z)-\frac{\pi}{\gamma},\,\, z\in\RR^N
\end{cases}
\end{equation}
and
\begin{equation}
\label{augg-eq9}
\bar v^*(x;s,z)\ge v^*(x;s+\frac{\pi}{\gamma},z)\quad \forall\,\, s\in\RR,\,\, x,z\in\RR^N.
\end{equation}
Hence, when $s\ge \bar s(z)-\frac{\pi}{\gamma}$, we have
\begin{align*}
u(1,0;\bar v^*(\cdot;s,z),z)&\geq u(1,0; v^*(\cdot;s+\frac{\pi}{\gamma},z),z)\qquad\qquad {\rm (by \,\, \eqref{augg-eq9})}\nonumber\\
&\ge  v(s+\frac{\pi}{\gamma},z)\qquad\qquad \qquad\qquad\quad {\rm (by \,\, \eqref{augg-eq7})}\nonumber\\
&=\bar v(s,z)\qquad\qquad\qquad \qquad\qquad\quad {\rm (by \,\,\eqref{augg-eq8})}\nonumber\\
&=\bar v^*((\kappa(\mu,z)-\tau(\mu,\gamma,z,B))\xi;s,(-\kappa(\mu,z)+\tau(\mu,\gamma,z,B))\xi+z).
\end{align*}
When $s<\bar s(z)-\frac{\pi}{\gamma}$,
$$
\bar v^*(x,s,z)\ge \bar v^*(x,\bar s(z)-\frac{\pi}{\gamma},z)\ge v^*(x,\bar s(z),z)
$$
and then
\begin{align*}
u(1,0;\bar v^*(\cdot;s,z),z)&\geq u(1,0; v^*(\cdot;\bar s(z),z),z)\qquad\qquad {\rm (by \,\, \eqref{augg-eq9})}\nonumber\\
&\ge  v(\bar s(z),z)\qquad\qquad \qquad\qquad\quad {\rm (by \,\, \eqref{augg-eq7})}\nonumber\\
&=\bar v(s,z)\qquad\qquad\qquad \qquad\qquad {\rm (by \,\,\eqref{augg-eq8})}\nonumber\\
&=\bar v^*((\kappa(\mu,z)-\tau(\mu,\gamma,z,B))\xi;s,(-\kappa(\mu,z)+\tau(\mu,\gamma,z,B))\xi+z).
\end{align*}
Therefore,
$$
u(1,0;\bar v^*(\cdot;s,z),z)\ge\bar v^*((\kappa(\mu,z)-\tau(\mu,\gamma,z,B))\xi;s,(-\kappa(\mu,z)+\tau(\mu,\gamma,z,B))\xi+z)
$$
for  all $s\in\RR$ and $z\in\RR^N$.

Let
$$
v_0(x;z)=\bar v(x\cdot \xi,x+z).
$$
Note that $\bar v(s,x)$ is non-increasing in $s$. Hence we have
\begin{align*}
u(1,x;v_0(\cdot;z),z)&=u(1,0;v_0(\cdot+x;z),x+z)\\
&= u(1,0;\bar v^*(\cdot;x\cdot\xi+ \kappa(\mu,x+z)-\tau(\gamma,x+z),x+z),x+z)\\
&\geq \bar v(x\cdot\xi+\kappa(\mu,x+z)-\tau(\mu,\gamma,x+z,B),x+z)\\
&\geq \bar v(x\cdot \xi-\lambda^{'}(\mu)+\epsilon_1,x+z)\quad\qquad {\rm (by}\quad \eqref{new-eq4})\\
&\geq \bar v\Big(x\cdot \xi-\frac{\lambda(\mu^{*}(\xi))}{\mu^{*}(\xi)}+2\epsilon_1,x+z\Big)\quad
{\rm (by}\quad \eqref{estimate-eq3})\\
&=v_0\Big(x-[\frac{\lambda(\mu^*(\xi))}{\mu^*(\xi)}-2\epsilon_1]\xi,[\frac{\lambda(\mu^*(\xi))}{\mu^*(\xi)}
-2\epsilon_1]\xi+z\Big)
\end{align*}
for $z\in\RR^N$. Let $\tilde
c^*(\xi)=\frac{\lambda(\mu^*(\xi))}{\mu^*(\xi)}-2\epsilon_1$. Then
$$
u(1,x;v_0(\cdot,z),z)\geq v_0(x-\tilde c^*(\xi)\xi,\tilde c^*(\xi)\xi+z)
$$
for all $z\in\RR^N$. We also have
\begin{align*}
u(2,x;v_0(\cdot,z),z)&\geq u(1,x;v_0(\cdot-\tilde c^*(\xi)\xi,\tilde c^*(\xi)\xi+z),
z)\\
&=u(1,x-\tilde c^*(\xi)\xi;v_0(\cdot,\tilde c^*(\xi)\xi+z),\tilde c^*(\xi)\xi
+z)\\
&\geq v_0(x-2\tilde c^*(\xi)\xi,2\tilde c^*(\xi)+z)
\end{align*}
for all $z\in\RR^N$. By induction, we have
$$
u(n,x;v_0(\cdot,z),z)\geq v_0(x-n\tilde c^*(\xi)\xi,n\tilde c^*(\xi)+z)
$$
for $n\geq 1$ and  $z\in\RR^N$. This together with Lemma \ref{main-lm1}  implies that
$$
c^*_{\inf}(\xi)\geq \tilde c^*(\xi)=\frac{\lambda(\mu^*(\xi))}{\mu^*(\xi)}-2\epsilon_1.
$$
Since $\epsilon_1$ is arbitrary,  \eqref{ineq2}  holds.
\end{proof}

\begin{proof}[Proof of Theorem \ref{spreading-thm1}]
 Fix $\xi\in S^{N-1}$. Put $\lambda(\mu)=\lambda_0(\xi,\mu,a_0)$, where
  $a_0(t,x)=f(t,x,0)$. By  Proposition \ref{PE-minimum-prop}, there is
$\mu^*=\mu^*(\xi)\in (0,\infty)$ such that
$$
\inf_{\mu>0}\frac{\lambda(\mu)}{\mu}=\frac{\lambda(\mu^*)}{\mu^*}.
$$
It is easy to see that $c^*(\xi)$ exists and $c^*(\xi)=\frac{\lambda(\mu^*)}{\mu^*}$  if and only if
$c_{\inf}^*(\xi)=c_{\sup}^*(\xi)=\frac{\lambda(\mu^*)}{\mu^*}$.

If  $\lambda_0(\xi,\mu,a_0)$ is the principal eigenvalue of
$-\p_t +\mathcal{K}_{\xi,\mu}-I+a_0(\cdot,\cdot)I$ for all $\mu$, then by Lemmas \ref{main-lm3} and \ref{main-lm4},
 we have $c^*(\xi)$ exists and
$c^*(\xi)=\inf_{\mu>0}\frac{\lambda(\mu)}{\mu}$.

In general, let $a^n(\cdot,\cdot)\in C^N(\RR\times \RR^N,\RR)\cap \mathcal{X}_p$ be such that
$a^n$ satisfies the vanishing condition in Proposition \ref{PE-sufficient-prop},
$$
a^n\geq a_0\quad {\rm for}\quad n\geq 1\quad {\rm and}\quad
\|a^n-a\|_{\mathcal{X}_p}\to 0\quad {\rm as}\quad n\to\infty.
$$
Then,
$$
\lambda_0(\xi,\mu,a^n)\to \lambda_0(\xi,\mu,a_0) \quad {\rm as}\quad n\to\infty.
$$
  Note that for any $M_0>0$,
 $$uf(t,x,u)\leq u (a^n(t,x)-\epsilon u)\quad {\rm for}\quad x\in\RR^N,\,\,  0\le u\le M_0,\,\, 0<\epsilon\ll 1.
 $$
 By Lemma \ref{main-lm3} and Proposition \ref{comparison-nonlinear-prop}, for any $u_0\in X^+(\xi)$ and $c>\inf_{\mu>0}\frac{\lambda_0(\xi,\mu,a^n)}{\mu}$,
 $$
 \lim_{x\cdot\xi\ge ct,t\to\infty} u(t,x;u_0)\le \lim_{x\cdot\xi\ge ct, t\to\infty} u^n(t,x;u_0)=0,
 $$
 where $u_n(t,x;u_0)$ is the solution of \eqref{main-shifted-eq} with $f(t,x,u)$ being replaced by $f^n(t,x,u)=a^n(t,x)-\epsilon u$ ($0<\epsilon\ll 1$).
 This implies that
 $$
 c^*_{\sup}(\xi)\leq \frac{\lambda_0(\xi,\mu,a^n)}{\mu} \quad \forall\,\, \mu>0,\,\, n\geq 1
 $$
 and then
 $$
 c^*_{\sup}(\xi)\leq \frac{\lambda_0(\xi,\mu,a_0)}{\mu}\quad \forall \mu>0.
 $$
 Therefore,
 \begin{equation}
 \label{speed-eq1}
 c^*_{\sup}(\xi)\leq \inf_{\mu>0} \frac{\lambda_0(\xi,\mu,a_0)}{\mu}.
 \end{equation}

 For any $\epsilon>0$, there is $\delta_0>0$ such that
 $$
 f(t,x,u)\geq f(t,x,0)-\epsilon\quad {\rm for}\quad t\in\RR,\,\, x\in\RR^N,\,\, 0<u<\delta_0.
 $$
 Let $a_n(\cdot,\cdot)\in C^N(\RR\times\RR^N)\cap \mathcal{X}_p$ be such that $a_n$ satisfies the vanishing condition in Proposition \ref{PE-sufficient-prop} and
 $$
 f(\cdot,\cdot,0)-2\epsilon \leq a_n(\cdot,\cdot)\leq f(\cdot,\cdot,0)-\epsilon\quad \forall n\geq 1.
 $$
 Note that
 $$
 uf(t,x,u)\ge u(a_n(t,x)-Mu)\quad \forall \,\, 0\le u\le \delta_0,\,\, M>0.
 $$
 Choose $M\ge \frac{\max_{t\in\RR,x\in\RR^N}a_n(t,x)}{\delta_0}$. By Lemma \ref{main-lm4} and Proposition \ref{comparison-nonlinear-prop},
 for any $u_0\in X^+(\xi)$ with $\sup_{x\in\RR^N}u_0(x)\le \delta_0$,
 $$
 \liminf_{x\cdot\xi\le ct, t\to\infty} u(t,x;u_0,z)\ge \liminf_{x\cdot\xi\le ct, t\to\infty} u_n(t,x;u_0,z)>0
 $$
 for any $c<\inf_{\mu>0}\frac{\lambda_0(\xi,\mu,a_n)}{\mu}$, where $u_n(t,x;u_0,z)$ is the solution of \eqref{main-shifted-eq}
 with $f(t,x,u)$ being replaced by $f_n(t,x,u)=a_n(t,x)-Mu$.
 This  implies that
 $$
 c_{\inf}^*(\xi)\ge
\inf_{\mu>0}\frac{\lambda_0(\xi,\mu,a_n)}{\mu}.
$$
 Thus,
 $$
 c^*_{\inf}(\xi)\geq \inf_{\mu>0}\frac{\lambda_0(\xi,\mu,a_0)-2\epsilon}{\mu}.
 $$
Letting $\epsilon\to 0$,  we have
\begin{equation}
\label{speed-eq2}
 c^*_{\inf}(\xi)\geq \inf_{\mu>0}\frac{\lambda_0(\xi,\mu,a_0)}{\mu}.
 \end{equation}

 By \eqref{speed-eq1} and \eqref{speed-eq2},
 $$c^*_{\sup}(\xi)=c^*_{\inf}(\xi)=\inf_{\mu>0}\frac{\lambda_0(\xi,\mu,a_0)}{\mu}.
 $$
 Hence $c^*(\xi)$ exists and
 $$
 c^*(\xi)=\inf_{\mu>0}\frac{\lambda_0(\xi,\mu,a_0)}{\mu}.
 $$

\end{proof}

\begin{proof}[Proof Theorem \ref{spreading-thm2}]
(1) It follows from the definition of $c^*(\pm\xi)$ directly.

(2) If follows from the arguments in \cite[Theorem E (2)]{ShZh2} and \cite[Theorem D (2)]{ShZh0}.

(3) and (4)  can be proved by the similar arguments as in \cite[Theorem E (3), (4)]{ShZh2}.
\end{proof}

\section{Traveling Wave Solutions}

In this section, we explore the existence  of traveling wave solutions of \eqref{main-eq} connecting
$0$ and $u^*$. Throughout this section, we assume (H1) and (H2).

\begin{definition}[Traveling wave solution]
\label{wave-def}
\begin{itemize}
\item[(1)]
An entire solution $u(t,x)$ of \eqref{main-eq} is called a {\rm traveling wave solution} connecting
$u^*(\cdot,\cdot)$ and $0$ and propagating in the direction of $\xi$ with speed $c$ if there is a bounded
 function $\Phi:\RR^N\times\RR\times  \RR^N\to \RR^+$ satisfying  that
 $\Phi(\cdot,\cdot,\cdot)$ is Lebesgue measurable,  $u(t,\cdot;\Phi(\cdot,0,z),z)$ exists
for all $t\in\RR$,
\begin{equation}
\label{wave-eq1} u(t,x)= u(t,x;\Phi(\cdot,0,0),0)=\Phi(x-ct\xi,t,ct\xi)\quad\forall t\in\RR,\,\, x\in\RR^N,
\vspace{-.05in}\end{equation}
\begin{equation}
\label{def-eq1}
u(t,x;\Phi(\cdot,0,z),z)=\Phi(x-ct\xi,t,z+ct\xi)\quad\forall t\in\RR,\,\, x,z\in\RR^N,
\end{equation}
\begin{equation}
\label{def-eq2}
\lim_{x\cdot\xi \to -\infty}\big(\Phi(x,t,z)-u^*(t,x+z)\big)=0,\quad \lim_{x\cdot\xi\to\infty}\Phi(x,t,z)=0
\end{equation}
uniformly in $(t,z)\in\RR\times \RR^N$,
\begin{equation}
\label{def-eq3}
\Phi(x,t,z-x)=\Phi(x^{'},t,z-x^{'})\quad \forall x,x^{'}\in\RR^N\,\, \text{with}\,\, x\cdot\xi=x^{'}\cdot\xi,
\end{equation}
and
\begin{equation}
\label{def-eq4}
\Phi(x,t+T,z)=\Phi(x,t,z+p_i{\bf e_i})=\Phi(x,t,z)\quad \forall x,z\in\RR^N.
\end{equation}

\item[(2)] A bounded  function $\Phi:\RR^N\times\RR\times \RR^N\to \RR^+$ is said to {\rm generate a traveling wave solution of
\eqref{main-eq} in the direction of $\xi$  with speed $c$} if it is Lebesgue measurable and  satisfies \eqref{def-eq1}-\eqref{def-eq4}.
\end{itemize}
\end{definition}

\begin{remark}
\label{wave-rk} Suppose that $u(t,x)=\Phi(x-ct\xi,t,c t\xi)$ is a
traveling wave solution of \eqref{main-eq} connecting $u^*(\cdot,\cdot)$
and $0$ and propagating in the direction of $\xi$ with speed $c$.
Then $u(t,x)$ can be written as
\begin{equation}
\label{wave-eq2}
 u(t,x)=\Psi(x\cdot\xi-ct,t,x)
\end{equation}
for some $\Psi:\RR\times\RR\times \RR^N\to \RR$ satisfying that
$\Psi(\eta,t+T,z)=\Psi(\eta,t,z+p_i{\bf e_i})=\Psi(\eta,t,z)$, $\lim_{\eta\to -\infty} \Psi(\eta,t,z)=u^*(t,z)$, and
$\lim_{\eta\to\infty}\Psi(\eta,t,z)=0$ uniformly in $(t,z)\in\RR\times \RR^N$. In fact, let
$\Psi(\eta,t,z)=\Phi(x,t,z-x)$
for $x\in\RR^N$ with $x\cdot\xi=\eta$. Observe that $\Psi(\eta,t,z)$
is well defined and has the above mentioned properties.
\end{remark}

For convenience, we introduce the following assumption:

\medskip
\noindent {\bf (H3)}
 {\it
 For
every $\xi\in S^{N-1}$ and   $\mu\geq 0$,
$\lambda_0(\xi,\mu,a_0)$ is the principal eigenvalue of $-\p_t+\mathcal{K}_{\xi,\mu}-I+a_0(\cdot,\cdot)I$,
 where $a_0(t,x)=f(t,x,0)$. }
\medskip

We now state the main results of this section.  For given $\xi\in S^{N-1}$ and $c>c^*(\xi)$, let $\mu\in (0,\mu^*(\xi))$ be such that
$$
c=\frac{\lambda_0(\xi,\mu,a_0)}{\mu}.
$$
Let $\phi(\mu,\cdot,\cdot)\in\mathcal{X}_p^+$ be the positive principal eigenfunction of
$-\p_t+\mathcal{K}_{\xi,\mu}-I+a_0(\cdot)I$ with $\|\phi(\mu,\cdot,\cdot)\|_{\mathcal{X}_p}=1$.

\begin{theorem}[Existence of traveling wave solutions]
\label{existence-thm}
Assume (H1)-(H3).
 For any $\xi\in S^{N-1}$ and $c> c^*(\xi)$, there is a bounded continuous function $\Phi:\RR^N\times\RR\times \RR^N\to \RR^+$ such that
  $\Phi(\cdot,\cdot,\cdot)$ generates a traveling
wave solution  connecting $u^*(\cdot,\cdot)$ and $0$ and propagating
in the direction of $\xi$ with speed $c$. Moreover,
$\ds\lim_{x\cdot\xi\to \infty}\frac{\Phi(x,t,z)}{e^{-\mu x\cdot\xi}\phi(\mu,t,x+z)}=1$
uniformly in $t\in\RR$ and $z\in\RR^N$.
\end{theorem}

\subsection{ Sub- and super-solutions}
\label{subsection-sub-super-solu}

In this subsection, we construct some sub- and super-solutions of \eqref{main-eq} to be used in the proof of Theorem \ref{existence-thm}.
 Throughout this subsection, we assume (H1)-(H3) and put $a_0(t,x)=f(t,x,0)$.

For given $\xi\in S^{N-1}$, let $\mu^*(\xi)$ be such that
\vspace{-.05in}$$
c^*(\xi)=\frac{\lambda_0(\xi,\mu^*(\xi),a_0)}{\mu^*(\xi)}.
\vspace{-.05in}$$
Fix $\xi\in S^{N-1}$ and $c>c^*(\xi)$. Let $0<\mu<\mu_1<\min\{2\mu,\mu^*(\xi)\}$ be such that
$
c=\frac{\lambda_0(\xi,\mu,a_0)}{\mu}
$
 and
$
\frac{\lambda_0(\xi,\mu,a_0)}{\mu}>\frac{\lambda_0(\xi,\mu_1,a_0)}{\mu_1}>c^*(\xi).
$
Put
$$
\phi_0(\cdot,\cdot)=\phi(0,\cdot,\cdot),
$$
and
 $$
 \phi(\cdot,\cdot)=\phi(\mu,\cdot,\cdot),\,\,\, \phi_1(\cdot,\cdot)=\phi(\mu_1,\cdot,\cdot).
 $$
If no confusion occurs, we may write $\lambda_0(\mu,\xi,a_0)$ as $\lambda(\mu)$.

For given $d>0$, let
\begin{equation*}
\underline v(t,x;z,d)= e^{-\mu (x\cdot\xi-ct)}\phi(t,x+z)-d
e^{-\mu_1(x\cdot\xi-ct)}\phi_1(t,x+z).
\end{equation*}
Observe that for given $0<b\ll 1$, there is $M>0$ such that
\begin{equation}
\label{u-underbar-eq1}
\begin{cases}
b\phi_0(t,x+z)\le \underline v(t,x;z,d) \quad \forall\,\, M-2\delta_0\le x\cdot\xi-ct\le M\cr
\underline v(t,x;z,d)>0\quad \forall\,\, x\cdot\xi -ct>M\cr
b\phi_0(t,x+z)\le e^{-\mu (x\cdot\xi-ct)}\phi(t,x+z)\quad \forall\,\, x\cdot\xi\le M+ct\cr
b\phi(t,x+z)\le u^*(t,x+z)\quad \forall x\in\RR^N,
\end{cases}
\end{equation}
where $\delta_0$ is such that ${\rm supp} \big(k(\cdot)\big)\subset \{z\in\RR^N\,|\, \|z\|<\delta_0\}$.
Let $b>0$ and $M>0$ be such that \eqref{u-underbar-eq1} holds and
\begin{equation}
\label{aux-sub-eq0}
\underline u(t,x;z,d,b)=\begin{cases}
\max\{b \phi_0(t,x+z), \underline v(t,x;z,d)\},\quad x\cdot\xi\le M+ct\cr
\underline v(t,x;z,d),\quad x\cdot\xi \ge M+ct.
\end{cases}
\end{equation}

\begin{proposition}
\label{sub-solution-prop1}
\begin{itemize}
 \item[(1)] There is $d^*>0$ such that for  any $z\in\RR^N$ and $d\ge d^*$,
  $\underline v(t,x;z,d)$ is a sub-solution of
  \eqref{main-shifted-eq}.

  \item[(2)] There is $b_0>0$ such that for any $0<b\le b_0$ and $z\in\RR^N$, \eqref{u-underbar-eq1} holds and $u(t,x;z):=b\phi_0(t,x+z)$ is the sub-solution of \eqref{main-shifted-eq}.

  \item[(3)] For $d\ge d^*$ and $0<b\le b_0$, $u(t,x;\underline u(0,\cdot;z,d,b),z)\ge \underline u(t,x;z,d,b)$ for $t\ge 0$.
\end{itemize}
\end{proposition}

\begin{proof}
(1) It follows from the similar arguments as in \cite[Propsotion 3.2]{ShZh1}.

(2) It follows from the similar arguments as in \cite[Proposition 3.3]{ShZh1}.

(3)  Let $\tilde w(t,x;z)=e^{Ct}\big(u(t,x;\underline u(0,\cdot;z,d,b),z)-\underline v(t,x;z,d)\big)$, where $C$ is some positive constant to
be determined later. Recall that
$u(t,x;\underline u(0,\cdot;z,d,b),z)$ is the solution of \eqref{main-shifted-eq} with
$u(0,x;\underline u(0,\cdot;z,d,b),z)=\underline u(0,x;z,d,b)$. Then
$$
\tilde w_t(t,x;z)\ge (\mathcal{K}_0 \tilde w)(t,x;z)+(-1+C+\tilde a(t,x,z))\tilde w(t,x;z)
$$
where
$$
(\mathcal{K}_0\tilde w)(t,x,;z)=\int_{\RR^N}k(y-x)\tilde w(t,y;z)dy
$$
 and
\begin{align*}
\tilde a(t,x,z)&= f(t,x,u(t,x;\underline u(0,\cdot;z,d,b),z))\\
&\quad +\underline v(t,x;z,d)\int_0^1 f_u(t,x, \tau (u(t,x;\underline u(0,\cdot;z,d,b),z)-\underline v(t,x;z,d)))d\tau.
\end{align*}
Hence
\begin{equation}
\label{aux-sub-eq1}
\tilde w(t,x;z)\ge \tilde w(0,x;z)+\int_0^ t\Big[ (\mathcal{K}_0 \tilde w)(s,x,z)+(-1+C+\tilde a(s,x;z))\tilde w(s,x;z)\Big] ds
\end{equation}
for all $x\in\RR^N$. Similarly,
let $\bar w(t,x;z)=e^{Ct}(u(t,x;z,\underline u(0,\cdot;z,d,b))-b\phi_0(t,x+z))$. Then
\begin{equation}
\label{aux-sub-eq2}
\bar w(t,x;z)\ge \bar w(0,x;z)+\int_0^t \Big[(\mathcal{K}_0\bar w)(s,x,z)+(-1+C+\bar a(s,x,z))\bar w(s,x;z) \Big] ds
\end{equation}
for $x\in\RR^N$, where
\begin{align*}
\bar a(t,x,z)&= f(t,x,u(t,x;\underline u(0,\cdot;z,d,b),z))\\
&\quad +b\phi_0(t,x+z)\int_0^1 f_u(t,x, \tau (u(t,x;\underline u(0,\cdot;z,d,b),z)-b\phi_0(t,x+z)))d\tau.
\end{align*}
Let $w(t,x;z)=e^{Ct}\big(u(t,x;\underline u(0,\cdot;z,d,b),z)-\underline u(t,x;z,d,b)\big)$.
Choose $C>0$ such that $-1+C+\tilde a(t,x,z)>0$ and $-1+C+\bar a(t,x,z)>0$. Note that
$$
w(t,x,z)=\begin{cases}
\tilde w(t,x;z)\quad {\rm for}\quad x\cdot\xi\ge M+ct\cr
\min\{\tilde w(t,x;z),\bar w(t,x;z)\}\quad {\rm for}\quad x\cdot\xi\le M+ct.
\end{cases}
$$
By \eqref{aux-sub-eq0}, \eqref{aux-sub-eq1}, and \eqref{aux-sub-eq2},
\begin{align*}
&\tilde w(t,x;z)\\
&\ge w(0,x;z)+\int_0^ t\Big[ (\mathcal{K}_0  w)(s,x,z)+(-1+C+\tilde a(s,x;z)) w(s,x;z)\Big] ds\quad {\rm for}\,\, x\in\RR^N
\end{align*}
and
\begin{align*}
&\bar w(t,x;z)\\
&\ge  w(0,x;z)+\int_0^t \Big[(\mathcal{K}_0  w)(s,x,z)+(-1+C+\bar a(s,x,z)) w(s,x;z) \Big] ds\quad {\rm for}\,\, x\cdot\xi\le M+ct.
\end{align*}
It then follows that
\begin{equation*}
 w(t,x;z)\ge w(0,x;z)+\int_0^ t\Big[ (\mathcal{K}_0  w)(s,x,z)+(-1+C+\tilde a(s,x;z)) w(s,x;z)\Big] ds\quad {\rm for}\,\, x\in\RR^N.
\end{equation*}
 By the arguments in \cite[Proposition 2.1(1)]{ShZh0}, we have $w(t,x;z)\ge 0$ for $t\ge 0$, $x,z\in\RR^N$, and then
$$u(t,x;\underline u(0,\cdot;z,d,b),z)\ge \underline u(t,x;z,d,b)
$$
for $t\ge 0$ and $x,z\in\RR^N$.
\end{proof}

Let
$$
\bar v(t,x;z)=e^{-\mu(x\cdot\xi-ct)}\phi(t,x+z)
$$
and
\begin{equation}
\label{u-abovebar-eq}
\bar u(t,x;z)=\min\{\bar v(t,x;z),u^{*}(t,x+z)\}.
\end{equation}

%We may write $\bar v(t,x;z)$ and $\bar u(t,x;z)$ for $\bar v(t,x;z)$ and
%$\bar u(t,x;z)$, respectively, if no confusion occurs.

\begin{proposition}
\label{sup-solution-prop1}
\begin{itemize}
\item[(1)]
 For any $z\in\RR^N$,
  $\bar v(t,x;z)$ is a super-solution of \eqref{main-shifted-eq}.

\item[(2)] $u(t,x;\bar u(0,\cdot;z),z)\le \bar u(t,x;z)$ for $t\ge 0$.

\end{itemize}
\end{proposition}

\begin{proof}
(1) It follows from the similar arguments as in \cite[Proposition 3.5]{ShZh1}.

(2) By comparison principle,
$$
u(t,x;\bar u(0,\cdot;z),z)\le \bar v(t,x;z)
$$
and
$$
u(t,x;\bar u(0,\cdot;z),z)\le  u^*(t,x+z)
$$
for $t\ge 0$.
(2) then follows.
\end{proof}

\begin{proposition}
\label{sup-sub-solution-prop}
 There is a constant $C$ such that for any $0<b\le b_0$ and $d\ge d^*$,
\begin{align}
\label{limit-at-negative-side-eq}
&\inf_{x\cdot\xi\le C,t\ge 0,z\in\RR^N}u(t,x+ct\xi;\bar u(0,\cdot;z),z)\nonumber\\
&\qquad\ge \inf_{x\cdot\xi\le C,t\ge 0,z\in\RR^N}u(t,x+ct\xi;\underline u(0,\cdot;z,d,b),z)\nonumber\\
&\qquad>0.
\end{align}
\end{proposition}

\begin{proof}
First of all, by \eqref{u-underbar-eq1}, \eqref{aux-sub-eq0}, and Propositions \ref{sub-solution-prop1} and \ref{sup-solution-prop1}, for any $t\ge 0$,
\begin{equation}
\label{bounds-at-both-sides-eq}
\underbar u(t,x;z,d,b)\le u(t,x;\underline u(0,\cdot; z,d,b),z)\le  u(t,x;\bar u(0,\cdot;z),z)\le \bar u(t,x;z).
\end{equation}
Observe that
\begin{align*}
\underline u(t,x+ct\xi;z,d,b)&=\max\{b\phi_0(t,x+ct\xi+z),\underline v(t,x+ct\xi;z,d)\}\,\, {\rm for}\,\, x\cdot\xi\le M\\
&\ge b\phi_0(t,x+ct\xi+z)\,\, {\rm for}\,\, x\cdot\xi\le M\\
&\ge \inf_{t\in\RR,x\in\RR^N} b\phi_0(t,x)\\
&>0.
\end{align*}
This together with \eqref{bounds-at-both-sides-eq} implies \eqref{limit-at-negative-side-eq}.
\end{proof}

\subsection{Traveling Wave Solutions}

In this subsection, we investigate the existence  of traveling wave solutions of \eqref{main-eq} and prove Theorem \ref{existence-thm}.
Throughout this section, we assume (H1)-(H3). For fixed $d\ge d^*$ and $0<b\le b_0$, put  $\underline u(t,x;z)=\underline u(t,x;z,d,b)$.

\begin{lemma}
\label{profile-exist-lm}
Let
$$
u^n (t,x,z)=u(t+nT,x+cnT\xi; \bar u (0,\cdot;z-cnT\xi),z-cnT\xi)
$$
and
$$
u_n(t,x,z)=u(t+nT,x+cnT\xi;\underbar u(0,\cdot;z-cnT\xi),z-cnT\xi).
$$
Then for any given bounded interval $I\subset\RR$, there is $N_0\in\NN$ such that
$u^n(t,x,z)$ is non-increasing in $n$ and $u_n(t,x,z)$ is non-decreasing in $n$ for $n\ge N_0$, $t\in I$, $x\in\RR^N$, $z\in\RR^N$.
\end{lemma}

\begin{proof}
First, observe that
$$
\bar u(T,x+cT\xi;z-cnT\xi)=\bar u(0,x;z-c(n-1)T\xi)\quad \forall\,\, n\ge 0.
$$
Hence for given $t\in\RR$ and $n\in\NN$ with $t+(n-1)T>0$,
\begin{align*}
&u^n(t,x,z)\\
&=u(t+nT,x+cnT\xi; \bar u (0,\cdot;z-cnT\xi),z-cnT\xi)\\
&=u(t+(n-1)T,x+cnT\xi;u(T,\cdot; \bar u (0,\cdot;z-cnT\xi),z-cnT\xi),z-cnT\xi)\\
&=u(t+(n-1)T,x+c(n-1)T\xi;u(T,\cdot+cT\xi;\bar u (0,\cdot;z-cnT\xi),z-cnT\xi);z-c(n-1)T\xi)\\
&\le u(t+(n-1)T,x+c(n-1)T\xi;\bar u(T,\cdot+cT\xi;z-cnT\xi);z-c(n-1)T\xi)\quad \text{(by Proposition \ref{sup-solution-prop1})}\\
&=u(t+(n-1)T,x+c(n-1)T\xi;\bar u(0,\cdot;z-c(n-1)T\xi),z-c(n-1)T\xi)\\
&=u^{n-1}(t,x,z).
\end{align*}

Similarly, we can prove that for given $t\in\RR$ and $n\in\NN$ with $t+(n-1)T>0$,
$$
u_n(t,x,z)\ge u_{n-1}(t,x,z).
$$
The proposition then follows.
\end{proof}

Let
$$
u^+(t,x,z)=\lim_{n\to\infty}u^n(t,x,z),
$$
$$
u^-(t,x,z)=\lim_{n\to\infty} u_n(t,x,z),
$$
and
$$
\Phi_0^\pm(x,z)=u^\pm(0,x,z).
$$
Then $u^+(t,x,z)$ and $\Phi_0^+(x,z)$ are upper semi-continuous in $t\in\RR$, $(x,z)\in\RR^N\times\RR^N$ and
$u^-(t,x,z)$ and $\Phi_0^-(x,z)$ are lower semi-continuous in $t\in\RR$, $(x,z)\in\RR^N\times\RR^N$.

\begin{lemma}
\label{existence-lm1}
For each $z\in\RR^N$,
$u^\pm(t,x,z)=u(t,x;\Phi_0^\pm(\cdot,z),z)$ for $t\in\RR$ and $x\in\RR^N$ and hence $u^\pm(t,x,z)$ are entire solution of \eqref{main-shifted-eq}.
\end{lemma}

\begin{proof} We prove the case that $u(t,x,z)=u^+(t,x,z)$.
First, note that
\begin{align*}
&u^n(t,x,z)\\
&=u(t,x+cnT\xi;u(nT,\cdot;\bar u(0,\cdot;z-cnT\xi),z-cnT\xi),z-cnT\xi)\\
&=u(t,x;u(nT,\cdot+cnT\xi;\bar u(0,\cdot;z-cnT\xi),z-cnT\xi),z)\\
&=u^n(0,x,z)\\
&\quad +\int_0^t \Big[\int_{\RR^N}k(y-x) u^n(\tau,y,z)dy-u^n(\tau,x,z)+u^n(\tau,x,z)f(\tau,x+z,u^n(\tau,x,z))\Big]d\tau
\end{align*}
Then by Lebesgue dominated convergence theorem,
\begin{align*}
u(t,x,z)=&\Phi_0^+(x,z)\\
&+ \int_0^t \Big[\int_{\RR^N}k(y-x) u(\tau,y,z)dy-u(\tau,x,z)+u(\tau,x,z)f(\tau,x+z,u(\tau,x,z))\Big]d\tau.
\end{align*}
This implies  that $u(t,x,z)=u(t,x;\Phi_0^+(\cdot,z),z)$ for all $t\in\RR$ and $x\in\RR^N$ and $u(t,x,z)$ is an entire solution of \eqref{main-shifted-eq}.
\end{proof}

\begin{proof}[Proof of Theorem \ref{existence-thm}]
Let
$$
\Phi^\pm(x,t,z)=u^\pm(t,x+ct\xi,z-ct\xi)
(=u(t,x+ct\xi;\Phi_0^\pm(\cdot,z-ct\xi),z-ct\xi)).
$$
It suffices to prove that $\Phi^\pm(x,t,z)$ generate  traveling wave solutions of \eqref{main-eq} with speed $c$ in the direction of $\xi$ and
$\Phi^+(t,x,z)=\Phi^-(t,x,z)$.

First of all, $u(t,x;\Phi^\pm(\cdot,0,z),z)=\Phi^\pm(x -ct\xi,t,z+ct\xi)$ follows directly from the definition of $\Phi^\pm(x,t,z)$.

Secondly, we prove that
$$\ds\lim_{x \cdot \xi -ct \to \infty} \frac{\Phi^\pm(x  -ct\xi,t,z+ct\xi)}
 {e^{-\mu (x\cdot\xi -ct)}\phi(t,x+z)}=1
 $$
  uniformly in $t\in\RR$ and  $z\in\RR^N$,
  which is equivalent to
  \begin{equation}
  \label{limit-eq1}
  \ds\lim_{x\cdot\xi\to\infty}
\frac{\Phi^\pm(x,t,z)}
{ e^{-\mu x\cdot\xi}\phi(t,x+z)}=1,
\end{equation}
 uniformly in $t\in\RR$ and  $z\in\RR^N$.
Note that
\begin{align}
\label{phi-eq0}
\underline v(t,x;z)&= e^{-\mu (x\cdot\xi -ct)}\phi(t,x+z)-d e^{-\mu_1(x\cdot\xi -ct)}\phi_1(t,x+z)\nonumber \\
&\leq u(t,x;\Phi^\pm(\cdot,0,z),z)\nonumber\\
&=\Phi^\pm(x  -ct\xi,t,z+ct\xi)\nonumber\\
& \leq \bar{v}(t,x;z)\nonumber\\
&=e^{-\mu(x \cdot \xi -ct)}\phi(t,x+z)
\end{align}
 for $t\in\RR$ and $x,z\in\RR^N$.
 \eqref{limit-eq1} then follows from  \eqref{phi-eq0}.

Thirdly, we  prove the periodicity of $\Phi^\pm(x,t,z)$ in $t$ and $x$.
Note  that
\begin{equation}
\label{phi-plus-eq1}
%\begin{cases}
\Phi^+(x,t,z)=\lim_{n\to\infty}u\Big(t+nT,x+cnT\xi+ct\xi;\bar u(0,\cdot;z-cnT\xi-ct\xi),z-cnT\xi-ct\xi\Big)
\end{equation}
\begin{equation}
\label{phi-minus-eq1}
\Phi^-(x,t,z)=\lim_{n\to\infty}u\Big(t+nT,x+cnT\xi+ct\xi;\underline u(0,\cdot;z-cnT\xi-ct\xi),z-cnT\xi-ct\xi\Big).
%\end{cases}
\end{equation}
By \eqref{phi-plus-eq1}, we have
\begin{align}
\label{phi-plus-eq3}
\Phi^+(x,T,z)&=\lim_{n\to\infty}u\Big((n+1)T,x+c(n+1)T\xi;\bar u(0,\cdot;z-c(n+1)T\xi),z-c(n+1)T\xi\Big)\nonumber\\
&=\lim_{n\to\infty}u\Big(nT,x+cnT\xi;\bar u(0,\cdot;z-cnT\xi),z-cnT\xi\Big)\nonumber\\
&=\Phi^+(x,0,z)
\end{align}
and
\begin{align}
\label{phi-plus-eq4}
&\Phi^+(x,t,z+p_i{\bf e_i})\nonumber\\
&=\lim_{n\to\infty}u\Big(t+nT,x+cnT\xi+ct\xi;\bar u(0,\cdot;z+p_i{\bf e_i}-cnT\xi-ct\xi),z+p_i{\bf e_i}-cnT\xi-ct\xi\Big)\nonumber\\
&=\lim_{n\to\infty}u\Big(t+nT,x+cnT\xi+ct\xi;\bar u(0,\cdot;z-cnT\xi-ct\xi),z-cnT\xi-ct\xi\Big)\nonumber\\
&=\Phi^+(x,t,z).
\end{align}
Moreover, for any $x,x^{'}\in\RR^N$ with $x\cdot\xi=x^{'}\cdot\xi$,
\begin{align}
\label{phi-plus-eq5}
&\Phi^+(x,t,z-x)\nonumber\\
&=\lim_{n\to\infty}u\Big(t+nT,x+cnT\xi+ct\xi;\bar u(0,\cdot;z-x-cnT\xi-ct\xi),z-x-cnT\xi-ct\xi\Big)\nonumber\\
&=\lim_{n\to\infty}u\Big(t+nT,cnT\xi+ct\xi;\bar u(0,\cdot+x;z-x-cnT\xi-ct\xi),z-cnT\xi-ct\xi\Big)\nonumber\\
&=\lim_{n\to\infty}u\Big(t+nT,cnT\xi+ct\xi;\bar u(0,\cdot+x^{'};z-x^{'}-cnT\xi-ct\xi),z-cnT\xi-ct\xi\Big)\nonumber\\
&=\lim_{n\to\infty}u\Big(t+nT,x^{'}+cnT\xi+ct\xi;\bar u(0,\cdot;z-x^{'}-cnT\xi-ct\xi),z-x^{'}-cnT\xi-ct\xi\Big)\nonumber\\
&=\Phi^+(x^{'},t,z-x^{'}).
\end{align}
Similarly, we have
\begin{equation}
\label{phi-minus-eq3}
\Phi^-(x,T,z)=\Phi^-(x,0,z),
\end{equation}
\begin{equation}
\label{phi-minus-eq4}
\Phi^-(x,t,z+p_i{\bf e_i})=\Phi^-(x,t,z)
\end{equation}
and for any $x,x^{'}\in\RR^N$ with $x\cdot\xi=x^{'}\cdot\xi$,
\begin{equation}
\label{phi-minus-eq5}
\Phi^-(x,t,z-x)=\Phi^-(x^{'},t,z-x^{'}).
\end{equation}

We now prove that
\begin{equation}
\label{limit-eq2}
\lim_{x\cdot\xi\to -\infty}(\Phi^\pm(x,t,z)-u^*(t,x+z))=0
\end{equation}
uniformly in $t\in\RR$ and $z\in\RR^N$. Note that there is $N_0\in\NN$ such that for $t\in [0,T]$ and  $n\ge N_0$,
\begin{align}
\label{left-limit-eq1}
u^*(t,x+z)&\ge u\Big(t+nT,x+cnT\xi+ct\xi;\bar u(0,\cdot;z-cnT\xi-ct\xi),z-cnT\xi-ct\xi\Big)\nonumber\\
&\ge\Phi^+(x,t,z)\nonumber\\
&\ge\Phi^-(x,t,z)\nonumber\\
&\ge u\Big(t+nT,x+cnT\xi+ct\xi;\underline u(0,\cdot;z-cnT\xi-ct\xi),z-cnT\xi-ct\xi\Big).
\end{align}
By  Proposition \ref{sup-sub-solution-prop}, there are $\sigma>0$ and $C\in\RR$ such that
\begin{align*}
&u(t+nT,x;\underline u(0,\cdot+cnT\xi+ct\xi;z-cnT\xi-ct\xi),z)\\
&=u\Big(t+nT,x+cnT\xi+ct\xi;\underline u(0,\cdot;z-cnT\xi-ct\xi),z-cnT\xi-ct\xi\Big)\\
&\ge \sigma
\end{align*}
for  $t\in [0,T],\,\, n\ge 1,\,\, x\cdot\xi\le C$. Then by Lemma \ref{main-lm1}, for any $\epsilon>0$ and $c^{'}<0$,  there is $N^*\in\NN$ with $N^*\ge N_0$  such that
\begin{align}
\label{left-limit-eq2}
&u\Big(t+N^*T,x+cN^*T\xi+ct\xi;\underline u(0,\cdot;z-cN^*T\xi-ct\xi),z-cN^*T\xi-ct\xi\Big)\nonumber\\
&u(t+N^*T,x;\underline u(0,\cdot+cN^*T\xi+ct\xi;z-cN^*T\xi-ct\xi),z)\nonumber\\
&\ge u^*(t+N^*T,x+z)-\epsilon
\end{align}
for $t\in [0,T]$ and $x\cdot\xi\le c^{'}(N^*+1)T$.
\eqref{limit-eq2} then follows from \eqref{left-limit-eq1}, \eqref{left-limit-eq2}, and the periodicity of $\Phi^\pm(x,t,z)$ in $t$.

By \eqref{limit-eq1}, \eqref{phi-plus-eq3}-\eqref{limit-eq2}, $\Phi^\pm(x,t,z)$ generate  traveling wave solutions of \eqref{main-eq} in the direction of $\xi$ with speed $c$.

Finally we prove that $\Phi^-(x,t,z)=\Phi(x,t,z)$, which implies that
$\Phi(x,t,z):=\Phi^+(x,t,z)$ is continuous and generates a traveling
wave solution of \eqref{main-eq} in the direction of $\xi$ with
speed $c$. To do so,  let
$$
\rho(t)=\inf\{\ln \alpha\,|\, \alpha\ge 1, \,
\frac{1}{\alpha}\Phi^-(x,t,z)\le\Phi^+(x,t,z)\le
\alpha\Phi^-(x,t,z),\,\, \forall\,\, x,z\in\RR^N\}.
$$
By \eqref{limit-eq1}, $\rho(t)$ is well defined. Moreover, there is $\alpha(t)\ge 1$ such that $\rho(t)=\ln\alpha(t)$ and
$$
\frac{1}{\alpha(t)}\Phi^-(x,t,z)\le\Phi^+(x,t,z)\le
\alpha(t)\Phi^-(x,t,z),\,\, \forall\,\, x,z\in\RR^N.
$$
First, we prove that $\rho(t)$ is non-increasing.  To this end, for given $u_0\in\tilde X$, let $u(t,x;\tilde u_0,z,s)$ be the solution of
\begin{equation}
\label{space-time-shifted-eq}
\frac{\p u}{\p t}=\int_{\RR^N} k(y-x)u(t,y)dy-u(t,x)+u(t,x) f(t+s,x+z,u(t,x)),\quad x\in\RR^N
\end{equation}
with $u(s,x;u_0,z,s)=u_0(x)$.
 For any $t_1<t_2$, we have
 $$
 u(t_2,x;\Phi_0^\pm(\cdot,z),z,0)=u(t_2,x;u(t_1,\cdot;\Phi_0^\pm(\cdot,z),z,0),z,t_1).
 $$
 Suppose that $\alpha_1\ge 1$ is such that $\rho(t_1)=\ln
 \alpha_1$ and
 $$
\frac{1}{\alpha_1}\Phi^-(x,t_1,z)\le \Phi^+(x,t_1,z)\le
\alpha_1\Phi^-(x,t_1,z)\quad \forall\,\, x,z\in\RR^N.
 $$
 Note that
 $$
 u(t_1,x;\Phi_0^\pm(\cdot,z),z,0)=\Phi^\pm(x-ct_1\xi,t_1,z+ct_1\xi).
 $$
Hence
$$
\frac{1}{\alpha_1}u(t_1,x;\Phi_0^-(\cdot,z),z,0)\le
u(t_1,x;\Phi_0^+(\cdot,z),z,0)\le
\alpha_1u(t_1,x;\Phi_0^-(\cdot,z),z,0)\quad \forall\,\,
x,z\in\RR^N.
$$
By  Proposition  \ref{comparison-nonlinear-prop} and (H1),
\begin{align*}
u(t_2,x;u(t_1,\cdot;\Phi_0^+(\cdot,z)z,0),z,t_1)&\le
u(t_2,x;\alpha_1u(t_1,x;\Phi_0^-(\cdot,z),z,0),z,t_1)\\
&\le \alpha_1u(t_2,x;u(t_1,\cdot;\Phi_0^-(\cdot,z),z,0),z,t_1)
\end{align*}
and
\begin{align*}
u(t_2,x;u(t_1,\cdot;\Phi_0^-(\cdot,z),z,0),z,t_1)&\le
u(t_2,x;\alpha_1u(t_1,x;\Phi_0^+(\cdot,z),z,0),z,t_1)\\
&\le \alpha_1u(t_2,x;u(t_1,\cdot;\Phi_0^+(\cdot,z,0),z,0),z,t_1)
\end{align*}
 for all $x,z\in\RR^N$. It then follows that
 $$
\frac{1}{\alpha_1}\Phi^-(x,t_2,z)\le \Phi^+(x,t_2,z)\le
\alpha_1\Phi^-(x,t_2,z)\quad \forall\,\, x,z\in\RR^N
 $$
 and hence
 $$
 \rho(t_2)\le \rho(t_1).
$$

 Next, we prove that $\rho(0)=0$. Assume that $\rho(0)>0$, then
there is $\alpha(0)>1$ such that $\rho(0)=\ln\alpha(0)$ and
$$
\frac{1}{\alpha(0)}\Phi^-(x,0,z)\le\Phi^+(0,x,z)\le \alpha(0)\Phi^-(x,0,z),\,\, \forall\,\, x\in\RR^N.
$$
By
\eqref{limit-eq1},
\begin{equation}
  \label{limit-eq1-1}
  \ds\lim_{x\cdot\xi\to\infty}
\frac{\Phi^\pm(x-ct\xi,t,z+ct\xi)} { e^{-\mu
(x\cdot\xi-ct)}\phi(t,x+z)}=1,
\end{equation}
 uniformly in $0\le t\le T$ and  $z\in\RR^N$. This implies that for
 any $\epsilon>0$ with $\frac{1+\epsilon}{1-\epsilon}<\alpha(0)$,
there is $M_\epsilon>0$ such that
$$
\frac{1-\epsilon}{1+\epsilon} \Phi^-(x-ct\xi,t,z+ct\xi)\le
\Phi^+(x-ct\xi,t,z+ct\xi)\le\frac{1+\epsilon}{1-\epsilon}\Phi^-(x-ct\xi,t,z+ct\xi)
$$
for $x\cdot\xi\ge M_\epsilon$, $0\le t\le T$ and $z\in\RR^N$. Note
that there is $\sigma_\epsilon>0$ such that for $x\cdot\xi\le
M_\epsilon$, $0\le t\le T$, and $z\in\RR^N$, there holds,
$$
\Phi^+(x-ct\xi,t,z+ct\xi)\ge \Phi^-(x-ct\xi,t,z+ct\xi)\ge \sigma_\epsilon.
$$
Let $\tilde u(t,x)=\alpha(0) u(t,x;\Phi_0^-(\cdot,z),z,0)$. Then there is $\delta_\epsilon>0$ such that
\begin{align*}
\tilde u_t&=(\mathcal{K}_0 \tilde u)(t,x)-\tilde u(t,x)+\tilde
u(t,x)f(t,x, u(t,x;\Phi_0^-(\cdot,z),z,0))\\
&\ge (\mathcal{K}_0 \tilde u)(t,x)-\tilde u(t,x)+\tilde
u(t,x)f(t,x, \tilde u(t,x))+\delta_\epsilon
\end{align*}
for $x\cdot\xi\le M_\epsilon$, $0\le t\le T$, and $z\in\RR^N$.
Let $\hat u(t,x)=u(t,x;\alpha(0)\Phi_0^-(\cdot,z),z,0)$. Note that
$$
\hat u_t(t,x)=(\mathcal{K}_0\hat u)(t,x)-\hat u(t,x)+\hat u(t,x)f(t,x,\hat u(t,x))
$$
for all $x\in\RR^N$ and $\tilde u(0,x)=\hat u(0,x)$ for $x\in\RR^N$, $\tilde u(t,x)\ge \hat u(t,x)$ for $x\cdot\xi\le M_\epsilon$.
Let $w(t,x)=\tilde u(t,x)-\hat u(t,x)$. Then
\begin{align*}
w_t(t,x)&\ge (\mathcal{K}_0w)(t,x)-w(t,x)+\tilde u(t,x)f(t,x,\tilde u(t,x))-\hat u(t,x)f(t,x,\hat u(t,x))+\delta_\epsilon\\
&\ge p(t,x)w(t,x)+\delta_\epsilon
\end{align*}
for $x\cdot\xi\le M_\epsilon$, where
$$
p(t,x)=-1+\Big[\tilde u(t,x)f(t,x,\tilde u(t,x))-\hat u(t,x)f(t,x,\hat u(t,x))\Big]/[\tilde u(t,x)-\hat u(t,x)].
$$
It then follows that
$$
w(t,x)\ge \int_0^t e^{\int_s^t (-1+p(\tau,x))d\tau}\delta_\epsilon ds
$$
for $x\cdot\xi\le M_\epsilon$.
This implies that there is $\tilde \delta_\epsilon>0$ such that
$$
\tilde u(T,x)\ge \hat u(T,x)+\tilde \delta_\epsilon
$$
for $x\cdot\xi\le M_\epsilon$. It follows that
$$
u(T,x;\Phi_0^+(\cdot,z),z,0)\le u(T,x;\alpha(0)\Phi_0^-(\cdot,z),z,0)\le \alpha(0)u(T,x;\Phi_0^-(\cdot,z),z,0)+\tilde \delta_\epsilon
$$
for $x\cdot\xi\le M_\epsilon$ and then there is $1\le \alpha_+<\alpha(0)$ such that
$$
u(T,x;\Phi_0^+(\cdot,z),z,0)\le \alpha_+ u(T,x;\Phi_0^-(\cdot,z),z,0)
$$
for $x\cdot\xi\le M_\epsilon$. Note that
$$
u(T,x;\Phi_0^-(\cdot,z),z,0)\le  u(T,x;\Phi_0^+(\cdot,z),z,0)\le \alpha_+ u(T,x;\Phi_0^+(\cdot,z),z,0)
$$
for any $x\in\RR^N$.
We therefore have
$$
\rho(T)\le \max\{\ln \alpha_+,\ln \frac{1+\epsilon}{1-\epsilon}\}<\ln\alpha(0)=\rho(0),
$$
this contradicts to $\rho(T)=\rho(0)$ for $n\ge 1$.
Hence $\rho(nT)=\rho(0)=0$. Then for any $0\le t\le T$,
we must have $0=\rho(0)\ge \rho(t)=\rho(T)=0$. Therefore,
$\rho(t)=0$  for all $t$ and
$$
\Phi^-(x,t,z)=\Phi^+(x,t,z)\quad \forall\,\, x,z\in\RR^N,\,\, t\in\RR.
$$
\end{proof}

\end{document}